\documentclass[12pt,a4paper]{amsart}

\usepackage{amssymb}
\usepackage[numeric, abbrev, nobysame]{amsrefs}
\usepackage{amscd}

\def\frak{\mathfrak}
\def\Bbb{\mathbb}
\def\Cal{\mathcal}

\let\phi\varphi

\newcommand{\x}{\times}
\renewcommand{\o}{\circ}

\newcommand{\al}{\alpha}
\newcommand{\be}{\beta}
\newcommand{\ga}{\gamma}

\newcommand{\ka}{\kappa}
\newcommand{\la}{\lambda}
\newcommand{\om}{\omega}
\newcommand{\ph}{\phi}
\newcommand{\ps}{\psi}
\renewcommand{\th}{\theta}
\newcommand{\si}{\sigma}
\newcommand{\ze}{\zeta}
\newcommand{\Ga}{\Gamma}
\newcommand{\La}{\Lambda}
\newcommand{\Ph}{\Phi}
\newcommand{\Ps}{\Psi}
\newcommand{\Om}{\Omega}

\newcommand{\tsum}{\textstyle\sum}

\newcommand{\Ad}{\operatorname{Ad}}
\newcommand{\Adb}{\operatorname{\underline{Ad}}}
\newcommand{\Adgr}{\operatorname{Ad_{\operatorname{gr}}}}

\newcommand{\hor}{\operatorname{hor}}
\newcommand{\im}{\operatorname{im}}

\newcommand{\gr}{\operatorname{gr}}
\newcommand{\id}{\operatorname{id}}
\newcommand{\pr}{\operatorname{pr}}
\newcommand{\ad}{\operatorname{ad}}

\newcommand{\Aut}{\operatorname{Aut}}

\newcommand{\tcn}{\widetilde{\mathcal N}}

\newcommand{\tcg}{\widetilde{\mathcal G}}

\newcounter{theorem}
\numberwithin{theorem}{section}
\numberwithin{equation}{section}

\newtheorem{thm}[theorem]{Theorem}
\newtheorem*{thm*}{Theorem \thesubsection}
\newtheorem{lemma}[theorem]{Lemma}
\newtheorem{prop}[theorem]{Proposition}

\newtheorem*{lemma*}{Lemma \thesubsection}
\newtheorem*{prop*}{Proposition \thesubsection}
\newtheorem*{cor*}{Corollary \thesubsection}

\theoremstyle{definition}
\newtheorem{definition}[theorem]{Definition}
\newtheorem*{definition*}{Definition \thesubsection}
\newtheorem{example}[theorem]{Example}
\newtheorem*{example*}{Example \thesubsection}
\theoremstyle{remark}
\newtheorem{remark}[theorem]{Remark}
\newtheorem*{remark*}{Remark \thesubsection}

\def\sideremark#1{\ifvmode\leavevmode\fi\vadjust{\vbox to0pt{\vss% the remark
 \hbox to 0pt{\hskip\hsize\hskip1em%                          will appear only
 \vbox{\hsize3cm\tiny\raggedright\pretolerance10000%          on the side
  \noindent #1\hfill}\hss}\vbox to8pt{\vfil}\vss}}}%
\begin{document}

\title{On canonical Cartan connections\\ associated to filtered
  G--structures} 

\author{Andreas \v Cap} 
\address{Faculty of Mathematics\\
University of Vienna\\
Oskar--Morgenstern--Platz 1\\
1090 Wien\\
Austria}
\email{Andreas.Cap@univie.ac.at}

\date{September 5, 2017}

\begin{abstract}
  A filtered manifold is a smooth manifold $M$ together with a
  filtration of the tangent bundle by smooth subbundles which is
  compatible with the Lie bracket of vector fields in a certain
  sense. The Lie bracket of vector fields then induces a bilinear
  operation on the associated graded of each tangent space of $M$
  making it into a nilpotent graded Lie algebra. Assuming that these
  symbol algebras are the same for all points, one obtains a natural
  frame bundle for the associated graded to the tangent bundle, and
  filtered G--structures are defined as reductions of structure group
  of this bundle. 

  Generalizing the case of parabolic geometries, this article is
  devoted to the question of whether a filtered G--structure of given
  type determines a canonical Cartan connection on an extended
  bundle. As for existence, the result are roughly as general as
  Morimoto's theorem from 1993, but it has several specific
  features. First, we allow for general candidates for a homogeneous
  model and a general version of normalization conditions. Second, the
  construction is entirely phrased in terms of Lie algebra valued
  forms and leads to an explicit characterization of the canonical
  Cartan connection. To verify that the procedure can be applied to a
  given type of filtered G--structures, only finite dimensional
  algebraic verifications have to be carried out. 
\end{abstract}
\subjclass[2010]{primary: 53B15, 53C10, 53C15 ; secondary: 53A40,
  53A55, 53C05, 53C17}
\keywords{canonical Cartan connection, Cartan geometry, filtered
  manifold, filtered G--structure}
\thanks{Supported by project P27072-N25 of the Austrian Science Fund
  (FWF) is gratefully acknowledged. I want to thank Boris Doubrov and
  Dennis The for very helpful discussions and for convincing me to
  write this article.}

\maketitle

\pagestyle{myheadings}\markboth{\v Cap}{Canonical Cartan connections}

\section{Introduction}\label{1}
Starting from E.\ Cartan's classical works in the early 20th century,
there is a long line of articles constructing canonical Cartan
connections associated to certain geometric structures. A Cartan
connection provides a description of the structure which is formally
similar to a certain homogeneous space called the \textit{homogeneous
  model} of the structure in question. Such constructions give rise to
a nice solution to the equivalence problem for the geometric structure
in question, since the curvature of a Cartan connection is known to be
a complete invariant. The classical constructions were often carried
out in the context of Cartan's method of equivalence, which gave them
a flavor of being difficult and involving extensive computations.

One of the structures dealt with by Cartan himself are strictly
pseudoconvex hypersurfaces in $\Bbb C^2$, for which a canonical Cartan
connection was constructed in \cite{Cartan:CR}. Generalizing this
result to higher dimensions was a hot topic in the late 1960's and
early 1970's, with the final results obtained (in the setting of
abstract CR structures) independently by N.~Tanaka in \cite{Tanaka:CR}
and by S.S.~Chern and J.~Moser in \cite{Chern-Moser}. While the
article by Chern and Moser was quickly considered as very important,
Tanaka's work received much less attention for a long time. (Of
course, it has to be mentioned here that \cite{Chern-Moser} does not
only contain the construction of a canonical Cartan connection, but
also deep and very influential results on normal forms for embedded CR
manifolds.)

Still, Tanaka's work has several very remarkable features. On the one
hand, it is more general, since only a very weak condition on
integrability of the CR structure is required. Moreover, Tanaka
noticed that the fact that the Lie algebra governing CR geometry is
simple is of crucial importance for the construction of a canonical
Cartan connection. Indeed, in the pioneering work
\cite{Tanaka:simple}, Tanaka showed that any parabolic subalgebra in a
simple Lie algebra determines a geometric structure for which
canonical Cartan connections can be constructed. The nature of these
results forced Tanaka's approach to be rather unusual in several
respects.

Most notably, the starting point for the construction was not provided
by the geometric structure to which one wants to associate a canonical
Cartan connection, but by the simple Lie algebra and parabolic subalgebra
describing the homogeneous model of the final Cartan geometry. Given a
parabolic subalgebra in a simple Lie algebra, there is a (reasonably
involved) description of an underlying geometric structure, to which
the procedure associates a canonical Cartan connection. Since simple
Lie algebras and parabolic subalgebras can be completely classified,
one ends up with a definite family of geometric structures for which
the procedure leads to canonical Cartan connections.

The geometric structures underlying the Cartan connections
corresponding to parabolic subalgebras were described in
\cite{Tanaka:simple} as standard G--structures satisfying certain
additional conditions. This makes them rather complicated to deal
with and even for simple examples like CR structures, the relation to
standard descriptions is not completely obvious. A big step towards a
simpler description of these structures was made in the works of
T.~Morimoto. His starting point was the concept of a \textit{filtered
  manifold}, i.e.~a smooth manifold $M$ endowed with a sequence of
nested smooth subbundles in the tangent bundle $TM$, which satisfy
certain (non--)integrability properties. It then turns out that the
Lie bracket of vector fields induces a bracket on the associated
graded vector space to each tangent space $T_xM$, making it into a
nilpotent graded Lie algebra, called the \textit{symbol algebra} of
the filtered manifold at the point $x$. These associated graded spaces
fit together to define a smooth vector bundle $\gr(TM)$ over $M$. Now
assume that for all points $x$, the symbol algebra is isomorphic to
fixed nilpotent graded Lie algebra $\frak m$. Then there is a natural
frame bundle for $\gr(TM)$ with structure group the group of
automorphisms of the graded Lie algebra $\frak m$. Similarly to the
classical case of G--structures, one can then consider reductions of
structure group of the natural frame bundle of $\gr(TM)$. Such
reductions are called \textit{filtered G--structures} and the
underlying structures obtained by Tanaka can be equivalently described
as such.

Morimoto initiated a comprehensive study of filtered manifolds, not
only from a geometric point of view, but also addressing topics like a
filtered version of prolongation of systems of differential equations
and, more generally, a filtered version of analysis. A substantial
body of results in that direction can be found in the long article
\cite{Morimoto:Cartan}. The last part of the article also contains a
general result on the existence of canonical Cartan connections
associated to filtered G--structures of finite type. This is embedded
into a general theory of (infinite) prolongations of geometric
structures using a non--commutative (semi--holonomic) version of frame
bundles. This prolongation procedure is used to determine the
homogeneous model of the geometry associated to a type of filtered
G--structures. Knowing this, there is an abstract definition of a
normalization condition needed to uniquely pin down the Cartan
connection. In view of this setup, applying Morimoto's result on
existence of canonical Cartan connections is a non--trivial task.

Starting from the 1980's, important developments in conformal geometry
gave new momentum to the theory. For example, the Fefferman--Graham
ambient metric from \cite{FG-amb} is motivated by CR geometry, and the
natural higher dimensional analog of (anti--)self--duality in
four--dimensional conformal geometry is provided by quaternionic
structures. Since all these structures admit canonical Cartan
connections this lead to renewed interest in Cartan geometries. The
developments in twistor theory and the Penrose transform as described
in the book \cite{Beastwood} opened up the perspective of studying
geometries related to all parabolic subalgebras in simple Lie
algebras. Initially being unaware of the results of Tanaka and
Morimoto, an independent procedure for constructing canonical Cartan
connections in this situation was found in \cite{Cap-Schichl}. This
initiated the general study of parabolic geometries, and after some
further developments the core of this theory was collected in
\cite{book}.

To construct a canonical Cartan connection one also has to extend the
principal bundle describing the underlying geometry to a bundle with a
larger structure group. In the constructions mentioned so far
(including the one in \cite{Cap-Schichl}) quite a lot of work goes
into the construction of this larger principal bundle. These
constructions are done in such a way that parts of the Cartan connection
can then be defined in a tautological way. On the other hand, simple
topological arguments show that in the parabolic case the principal
bundle on which the Cartan connection is defined has to be a trivial
extension of the bundle describing the filtered G--structure. 

Starting from this observation, yet another independent construction
of the canonical Cartan connections in the parabolic cases was given
in Section 3.1 of \cite{book} along the following lines. One directly
defines the Cartan bundle as a trivial extension of the bundle
describing the underlying structure. Next, one shows that there is a
Cartan connection on this extended bundle, which induces the
underlying geometric structure. This involves making choices, so it is
not canonical at all. Then one shows that any Cartan connection can be
modified to one that satisfies an appropriate normalization condition
without changing the underlying structure. Finally, one proves that a
normal Cartan geometry is uniquely determined up to isomorphism by the
underlying structure. In all that, the main role is played by the
algebraic properties of the normalization condition.

\medskip

The aim of this article is to generalize the construction of canonical
Cartan connections from \cite{book} beyond the parabolic case. The
main ingredient that is needed is a normalization condition with
appropriate algebraic properties. These properties are closely similar
to the so--called ``condition (C)'' from \cite{Morimoto:Cartan}. As
far as I can see (given the question of determining the group
governing the geometry as discussed above) this result should
essentially cover the same cases as Morimoto's result on Cartan
connections. However, there are some distinctive features of the
approach taken here:

\noindent
$\bullet$ Careful separation between geometry and algebra (point--wise
issues): In constructions of Cartan connections via the method of
equivalence, a typical part are step--by--step constructions of
adapted coframes. This is phrased in the language of differential
forms, but the conditions for a coframe to be adapted usually are
point--wise. Taking exterior derivatives exhibits consequences of the
conditions imposed so far, which then are used in the further steps of
the process. These frequent changes between geometry and point--wise
conditions often make it hard to understand what is going
on. Moreover, the point--wise conditions are often best expressed in
the language of linear algebra or of representation theory, which is
not as easy to use in the language of differential forms. In the
method of equivalence, this mix of geometric and point--wise
considerations is partly unavoidable, in particular if the process is
used to exhibit certain subclasses of geometries, which then require
separate treatment.

In this article, we carefully separate the geometric constructions
used to obtain canonical Cartan connections from the algebraic
background information needed in the construction. We work with a
uniform structure from the beginning. To show that the procedure
applies to a given type of geometric structure, only algebraic (finite
dimensional) verifications have to be carried out. Once these
verifications have been done, the universal constructions in this
article lead to canonical Cartan connections.

\noindent
$\bullet$ The starting point is a candidate for a homogeneous model
rather than a filtered geometric structure. Given a homogeneous space
$G/P$ we describe the algebraic data needed to obtain a $G$--invariant
filtered geometric structure on $G/P$. These can be phrased as a
filtration on the Lie algebra $\frak g$ of $G$ which has to be
compatible with the subgroup $P$ in a certain sense. One may then
forget about the group $G$ and just consider $\frak g$ and $P$. This
leads to the concept of an \textit{admissible pair}, see Definition
\ref{def2.2}. In particular, the filtration defines a closed subgroup
$P_+\subset P$ such that $G_0:=P/P_+$ is the structure group of the
underlying filtered geometric structure.

For an admissible pair $(\frak g,P)$, it turns out that any regular
Cartan connection of type $(\frak g,P)$ determines an underlying
filtered $G_0$--structure. If this structure determines a canonical
Cartan connection of type $(\frak g,P)$, then certainly $\frak g$ must
be the full Lie algebra of its infinitesimal automorphisms. This can
be phrased as the (purely algebraic) condition that the associated
graded Lie algebra $\gr(\frak g)$ is the full prolongation of its
non--positive part, see Definition \ref{def2.5}.

If one wants to start from the filtered $G_0$--structure instead
(i.e.~if no candidate for a homogeneous model is available), then it
is possible to build up such a candidate via Tanaka prolongation, see
Example (4) in Section \ref{2.6}. This however only determines the
associated graded Lie algebra to $\frak g$ and there may be several
possible choices of an admissible pair $(\frak g,P)$ inducing this
associated graded. For all these choices, the condition on the full
prolongation is then satisfied automatically. 

\noindent
$\bullet$ We use a general concept of normalization conditions which
extracts the essential properties that are needed. Imposing a
normalization condition on the curvature of the Cartan connection is
always necessary to ensure that the connection is uniquely determined
by the underlying filtered $G_0$--structure. For an admissible pair
$(\frak g,P)$ satisfying the condition on the full prolongation, a
good choice of normalization condition is the only additional
ingredient needed in order to get the machinery developed in this
article going. In the most general version, such a condition is
described by a linear subspace in a certain vector space, with
requirements detailed in Definition \ref{def-norm}. We show how such
normalization conditions can be obtained from a
\textit{codifferential}, but this is already a special case. Inner
products with certain invariance properties, which are the central
requirement in the construction of canonical Cartan connections in
\cite{Alekseevsky-David} can be used to construct codifferentials, but
only play an auxiliary role.

\noindent
$\bullet$ We obtain an explicit characterization of the canonical
Cartan connection, which leads to strong uniqueness results. For many
of the constructions available in the literature, uniqueness of
canonical Cartan connections follows from the naturality of the
construction used to obtain them. While this is a perfectly legitimate
argument, such an approach makes it difficult to compare the results
of different constructions of canonical Cartan connections. The
uniqueness results we prove here are of completely different nature. A
normalization condition of the form we use singles out a subspace in
the space of $\frak g$--valued two--forms on any Cartan geometry of
the given type. Such a geometry then is called \textit{normal} if the
curvature of the Cartan connection lies in this subspace. The basic
uniqueness result we prove in Theorem \ref{thm4.5} is that if two
normal regular Cartan geometries have isomorphic underlying filtered
$G_0$--structures, then they are themselves isomorphic. So to compare
to other constructions one just has to prove that these other
constructions lead to normal Cartan connections.

\noindent
$\bullet$ We develop a general concept of essential curvature
components. Having chosen a normalization condition, a general notion
of a \textit{negligible submodule} is given in Definition
\ref{def-neg}. Such a submodule defines a subspace in the space of
normal $\frak g$--valued two forms. We prove that projecting the
curvature to the quotient by this subspace one still obtains a
complete obstruction to local flatness. This generalizes the concept
of harmonic curvature used for parabolic geometries. In particular, we
show that a codifferential (in the sense of Definition
\ref{def-codiff}) automatically gives rise not only to a normalization
condition but also to a maximal negligible submodule.

\noindent
$\bullet$ The full construction is done in the language of $\frak
g$--valued forms on the Cartan bundle. This is in sharp contrast to
the construction in Section 3.1 of \cite{book} in which rather subtle
constructions with associated vector bundles play a crucial
role. Staying on the principal bundle, however, makes it necessary to
carefully distinguish between filtered objects and associated graded
objects on the level of the linear algebra background of the
construction. In particular, it will be important to carefully
distinguish between the filtered Lie algebra $\frak g$ and its
associated graded $\gr(\frak g)$, even though in many cases of
interest these happen to be isomorphic.

\medskip

The actual motivation for working out this general construction of
canonical Cartan connections was the joint article \cite{CDT} with
B.~Doubrov and D.~The, which uses Cartan connections canonically
associated to (systems of) ODEs. These applications depend on the
explicit characterization of normal Cartan geometries and the strong
algebraic properties or normalization conditions introduced here, which
are not available for earlier versions of canonical Cartan connections
associated to (systems of) ODEs. The algebras underlying these cases
are far from being parabolic, see part 3 of Example \ref{2.6}. The
construction of a normalization condition for this case is sketched in
part 3 of Example \ref{3.4}. The article \cite{CDT} contains the
complete algebraic verifications needed to apply the theory developed
here in this family of cases.

\smallskip 

It should be remarked here that many of the algebraic subtleties
described above can be avoided if one is willing to only construct an
absolute parallelism on an extended bundle rather than a Cartan
connection. This removes the necessity of requiring invariance or
equivariancy properties of normalization conditions, which may allow
the construction of canonical absolute parallelisms associated to
structures for which no canonical Cartan connections exist. General
constructions of such parallelisms again go back to work of Tanaka,
see \cite{Tanaka:prolon}. A simpler, complete construction for
filtered $G_0$--structures of finite type in modern language can be
found in \cite{Zelenko}. A canonical absolute parallelism still gives
rise to a complete set of invariants and thus a solution to the
equivalence problem. But already in this aspect, it seems much more
difficult to give a geometric interpretation of the resulting
invariants then in the case of a Cartan connection. Moreover, many
geometric tools are available for Cartan geometries (see
\cite{Sharpe}) or at least for the subclass of parabolic geometries
(see \cite{book}). While it seems plausible that many of the latter
tools can be generalized to larger classes of Cartan geometries, it
seems very hard to extend even basic geometric tools to the case of
absolute parallelisms. Thus, I believe that it is worthwhile to try to
obtain Cartan connections whenever possible.

\smallskip

To conclude this introduction, let us describe the contents of the
individual sections of the article. Section \ref{2} deals with the
algebraic ingredients needed to get an infinitesimal homogeneous model
for a filtered G--structure. We discuss the notion of an admissible
pair $(\frak g,P)$ and the condition that $\gr(\frak g)$ is the full
prolongation of of its non--positive part. We then define regular
Cartan geometries of type $(\frak g,P)$ and show in Theorem
\ref{thm2.4} that any such Cartan geometry determines an underlying
filtered $G_0$--structure, where $G_0=P/P_+$. In the end of the
Section, several examples are discussed in detail.

Section \ref{3} is devoted to the study of normalization conditions
and negligible submodules. Again, several examples are discussed in
the end of the section. In Section \ref{4} we start by setting up the
necessary background on $\frak g$--valued differential forms on
principal $P$--bundles and then define the covariant exterior
derivative on $\frak g$--valued forms induced by a Cartan
connection. Given a normalization condition, we define normal Cartan
connections. If we have also given a negligible submodule, we define
essential curvature and prove in Proposition \ref{prop4.3} that
vanishing of the essential curvature implies vanishing of the
curvature.

The first crucial result is Theorem \ref{thm4.4} on normalizing Cartan
connections. This only needs an appropriate normalization condition
for an admissible pair $(\frak g,P)$. Given a regular Cartan
connection $\om$ on a principal $P$--bundle $\Cal G\to M$, we show
that there is a regular normal Cartan connection $\hat\om$ on $\Cal
G$, which induces the same underlying filtered $G_0$--structure as
$\om$. The second crucial result is Theorem \ref{thm4.5}, which shows
that if two regular normal Cartan connections $\om$ and $\hat\om$ on a
principal $P$--bundle $\Cal G$ induce the same underlying structure,
then they are related by an automorphism of $\Cal G$ inducing the
identity on the underlying filtered $G_0$--structure.

To obtain a result on canonical Cartan connections associated to
filtered $G_0$--structures, one more ingredient is needed. This
concerns the algebraic and topological structure of the group $P$ and
is spelled out in Definition \ref{def4.6}. Assuming this condition, we
prove our final result in Theorem \ref{thm4.6}, namely that there is
an equivalence of categories between regular normal Cartan geometries
of type $(\frak g,P)$ and filtered $G_0$--structures. The key issue in
the proof is that the condition on $P$ suffices to show that the
bundle defining a filtered $G_0$--structure can always be extended to
a principal $P$--bundle. Making choices, one constructs a regular
Cartan connection on the extended bundle, which induces the given
filtered $G_0$--structure, and this can then be normalized. On the
other hand, the condition also implies that any morphism between the
underlying structures of two regular normal Cartan geometries lifts to
a morphism of the Cartan bundles, which then can be converted into a
morphism of Cartan geometries using the uniqueness result from Theorem
\ref{thm4.5}.

\section{Infinitesimal homogeneous models}\label{2}
We first recall the filtered version of a reduction of structure
group of the frame bundle, which, for compatibility with later
notation, we call a filtered $G_0$--structure. Next, we study the data
needed to define a $G$--invariant filtered $G_0$--structure on a
homogeneous space $G/P$ for an appropriate quotient group $G_0$ of
$P$. These data can all be phrased in terms of the Lie algebra $\frak
g$ of $G$ and the group $P$ only, which leads to the concept of an
admissible pair $(\frak g,P)$. Given such a pair, there is a natural
concept of Cartan geometry of type $(\frak g,P)$ and a notion of
regularity for such a geometry. We show that any regular Cartan
geometry of type $(\frak g,P)$ induces an underlying filtered
$G_0$--structure. The fact that for a homogeneous space $G/P$ as
above, the group $G$ is the full automorphism group of the
corresponding filtered $G_0$--structure has algebraic consequences,
which again can be phrased entirely in Lie algebraic terms. Since this
must evidently be the case if there is a canonical Cartan connection
of type $(\frak g,P)$ associated to filtered $G_0$--structures, this
gives rise to a necessary condition for existence such canonical
Cartan connections. We call an admissible pair $(\frak g,P)$ an
\textit{infinitesimal homogeneous model} for filtered
$G_0$--structures if this condition is satisfied.

\subsection{Filtered $G_0$--structures}\label{2.1} 
Recall that a \textit{filtered manifold} is a smooth manifold $M$
together with a filtration of the tangent bundle, which we write as
$$
TM=T^{-\mu}M\supset T^{-\mu+1}M\supset\dots\supset T^{-2}M\supset
T^{-1}M 
$$
such that for sections $\xi\in\Ga(T^iM)$ and $\eta\in\Ga(T^jM)$ the
Lie bracket $[\xi,\eta]$ is a section of $T^{i+j}M$. We call
$\mu\in\Bbb N$ the \textit{depth} of the filtration and follow the usual
convention that $T^{\ell} M=TM$ for all $\ell<-\mu$ and $T^\ell
M=M\x\{0\}$ for $\ell\geq 0$.

The \textit{associated graded} to the tangent bundle is then defined
as the bundle $\gr(TM)=\oplus_{i=-\mu}^{-1}\gr_i(TM)$, with
$\gr_i(TM)=T^iM/T^{i+1}M$. In particular, the fiber of $\gr_i(TM)$
over a point $x\in M$ simply is the quotient $T_x^iM/T_x^{i+1}M$, so
this is $\gr_i(T_xM)$. If necessary, we put $\gr_i(TM)=M\x\{0\}$ for
$i<-\mu$ and $i\geq 0$. Further, we denote by
$q_i(x):T^i_xM\to\gr_i(T_xM)$ the natural quotient map, so we get a
vector bundle map $q_i:T^iM\to\gr_i(TM)$.

Fix a point $x\in M$ and consider the operator
$\Ga(T^iM)\x\Ga(T^jM)\to\gr_{i+j}(T_xM)$ defined by $(\xi,\eta)\mapsto
q_{i+j}([\xi,\eta](x))$, which is well defined by definition of a
filtered manifold. Since $i,j\geq i+j+1$ a short computation using the
definition of a filtered manifold once more shows that
$q_{i+j}([\xi,\eta](x))$ depends only on the values of $\xi$ and
$\eta$ at $x$ and in fact only on their classes in $\gr_i(T_xM)$ and
$\gr_j(T_xM)$, respectively. Hence we get a well defined bilinear map
$\gr_i(T_xM)\x\gr_j(T_xM)\to\gr_{i+j}(T_xM)$. Collecting these maps
for different values of $i$ and $j$, we obtain a bilinear map $\Cal
L_x:\gr(T_xM)\x\gr(T_xM)\to\gr(T_xM)$, called the \textit{Levi
  bracket} at $x$. The properties of the Lie bracket of vector fields
readily imply that this operation makes $\gr(T_xM)$ into a graded
Lie algebra, which has to be nilpotent since the grading has finite
length. This is called the \textit{symbol algebra} of the filtered
manifold at $x$. Of course, we can collect the Levi brackets at the
individual points into a bilinear bundle map $\Cal
L:\gr(TM)\x\gr(TM)\to\gr(TM)$, the \textit{Levi bracket}.

A standing assumption on filtered manifolds that we will make is that
they are of constant type, i.e.~that the symbol algebras at all points
are isomorphic. More precisely, we require that the symbol algebras
form a locally trivial bundle of graded Lie algebras modelled on a
fixed nilpotent graded Lie algebra $\frak m=\oplus_{i=-\mu}^{-1}\frak
m_i$. This means that for each $x\in M$, we find an open neighborhood
$U\subset M$ of $x$ and local trivializations $\gr_i(TM)|_U\to
U\x\frak m_i$ for each $i=-\mu,\dots,-1$ such that for all $y\in U$
the corresponding isomorphisms $\ph_i:\gr_i(T_yM)\to\frak m_i$ have
the property that $\ph_{i+j}(\Cal L_y(u,v))=[\ph_i(u),\ph_j(v)]$ for
all $i$ and $j$, all $u\in\gr_i(T_yM)$ and $v\in\gr_j(T_yM)$. We then
say that the filtered manifold $(M,\{T^iM\})$ is \textit{regular of
  type $\frak m$}.

Now consider the group $GL(\frak m)$ of linear automorphisms of $\frak
m$. The group $\Aut_{\gr}(\frak m)$ of automorphisms of the Lie algebra
$\frak m$, which in addition preserve the grading of $\frak m$, clearly
is a closed subgroup in $GL(\frak m)$ and thus a Lie group. It is a
well known fact from Lie theory that the Lie algebra of this group is
$\frak{der}_{\gr}(\frak m)$, the space of all linear maps $\al:\frak
m\to\frak m$ which preserve the grading and are derivations in the
sense that they satisfy $\al([X,Y])=[\al(X),Y]+[X,\al(Y)]$ for all
$X,Y\in\frak m$.

\begin{prop}\label{prop2.1}
  Let $\frak m=\oplus_{i=-\mu}^{-1}\frak m_i$ be a finite dimensional
  nilpotent graded Lie algebra. Then for any filtered manifold
  $(M,\{T^iM\}_{i=-\mu}^{-1})$ that is regular of type $\frak m$, the
  bundle $\gr(TM)$ admits a canonical frame bundle $\Cal PM$ that is a
  principal bundle with structure group $\Aut_{\gr}(\frak m)$.
\end{prop} 
\begin{proof}
  Given $x\in M$, one defines $\Cal P_xM$ to be the set of all
  isomorphisms $\frak m\to\gr(T_xM)$ of graded Lie algebras. Then one
  defines $\Cal PM$ to be the disjoint union of the $\Cal P_xM$,
  endowed with the obvious projection $p:\Cal PM\to M$. Fixing one
  element of $\Cal P_x$ and composing it from the right by elements
  from $\Aut_{\gr}(\frak m)$ identifies $\Aut_{\gr}(\frak m)$ with $\Cal
  P_xM$. Taking an open subset $U$ as in the definition of constant
  type and doing this construction in each point, one obtains a
  bijection $p^{-1}(U)\to U\x\Aut_{\gr}(\frak m)$. It is now routine to
  use this to define a topology on $\Cal PM$ and make it into a
  principal bundle over $M$.
\end{proof}

The bundle $\Cal PM$ is the perfect analog of the linear frame bundle
of a smooth manifold in the setting of filtered manifolds of constant
type. Indeed, the linear frame bundle occurs as a special case. If one
takes the filtration of $M$ to be trivial, i.e.~$\mu=1$ and
$T^{-1}M=TM$, then one just gets a smooth manifold and this is regular
of type the abelian Lie algebra $\Bbb R^n$ with trivial grading. Of
course, the construction from Proposition \ref{prop2.1} then just
recovers the usual linear frame bundle with structure group $GL(n,\Bbb
R)$. Hence the following definition generalizes the usual concept of
G--structures.

\begin{definition}\label{def2.1} 
  Fix a nilpotent graded Lie algebra $\frak m$ and let $G_0$ be a Lie
  group endowed with a fixed infinitesimally injective homomorphism
  $\be:G_0\to\Aut_{\gr}(\frak m)$. Then a \textit{filtered
    $G_0$--structure} over a filtered manifold $M$ which is regular of
  type $\frak m$ is a reduction of structure group of the natural
  frame bundle $\Cal PM$ for $\gr(TM)$ to the group $G_0$. More
  explicitly, this is given by a principal $G_0$--bundle $\Cal G_0\to
  M$ and a smooth bundle map $\Ph:\Cal G_0\to\Cal PM$ that covers the
  identity on $M$ and is equivariant in the sense that $\Ph(u\cdot
  g)=\Ph(u)\cdot\be(g)$ for all $u\in\Cal G_0$ and $g\in G_0$.
\end{definition}

There is an obvious concept of morphisms in this setting. For a local
diffeomorphism $f$ between filtered manifolds $M$ and $\tilde M$,
which both are regular of type $\frak m$, there is an obvious concept
of being \textit{filtration preserving}. We just require that for each
point $x\in M$ the tangent map $T_xf:T_xM\to T_{f(x)}\tilde M$ is
compatible with the filtrations on two spaces. This implies that $T_x
f$ induces a linear isomorphism $\gr(T_xM)\to\gr(T_{f(x)}\tilde M)$
and it is easy to verify that this map is compatible with the
Levi--brackets. Hence there is an induced principal bundle map $\Cal
Pf:\Cal PM\to\Cal P\tilde M$ with base map $f$. Given a filtered
$G_0$--structures $\Cal G_0\to M$ defined by $\Ph:\Cal G_0\to\Cal PM$
and likewise for $\tilde M$, a morphism of filtered $G_0$--structures
is a principal bundle map $F:\Cal G_0\to\tcg_0$ such that $\tilde\Ph\o
F=\Cal Pf\o\Ph$ (which in particular implies that $F$ has base map
$f$). In the case that $G_0$ is a subgroup of $\Aut_{\gr}(\frak m)$, we
can view $\Cal G_0$ as a subbundle of $\Cal PM$, and we must have
$F=\Cal Pf|_{\Cal G_0}$, so the main condition is that $\Cal Pf(\Cal
G_0)\subset\tcg_0$.

\begin{remark}\label{rem2.1}
  (1) In most cases of interest, $G_0$ will simply be a subgroup of
  $\Aut_{\gr}(\frak m)$. The slightly more general setup is chosen to
  allow structures analogous to spin--structures in Riemannian
  geometry. In any case, infinitesimal injectivity implies that the
  derivative $\be'$ of $\be$ defines an isomorphism from $\frak g_0$
  onto a Lie subalgebra of $\mathfrak{der}_{\gr}(\frak m)$, so at least
  we will usually view $\frak g_0$ as a Lie subalgebra in there.

  (2) One gets even closer to the classical picture in the case that
  $\frak m$ is \textit{fundamental}, which means that it is generated
  by $\frak m_{-1}$ as a Lie algebra. This readily implies that any
  automorphism of the graded Lie algebra $\frak m$ is uniquely
  determined by its restriction to $\frak m_{-1}$. Consequently,
  $\Aut_{\gr}(\frak m)\subset GL(\frak m_{-1})$ and $\Cal PM$ can be
  viewed as a subbundle of the linear frame bundle of the vector
  bundle $T^{-1}M$. So a reduction $\Cal G_0\to\Cal PM$ can be
  interpreted as an additional reduction of structure group of
  $T^{-1}M$.

  (3) Similarly to the classical case, reductions of structure group
  can also be characterized by a filtered analog of a soldering
  form. To describe this, observe first that there is an induced
  filtration of $T\Cal G_0$. One simply defines $T^i\Cal G_0$ as the
  pre--image of $T^iM$ for $i<0$ and as the vertical subbundle for
  $i=0$. These subbundles are easily seen to be invariant under the
  principal right action. Now for each $i<0$, one obtains a
  ``differential form'' $\th_i$ which is only defined on the subbundle
  $T^i\Cal G_0$ and has values in $\frak m_i$, such that the
  point--wise kernel coincides with $T^{i+1}\Cal G_0$. Moreover these
  forms are equivariant for the principal right action and the
  representation of $G_0$ on $\frak m_i$ induced by the homomorphism
  $\be$. Conversely, one can construct a homomorphism to the frame
  bundle from such a family of partially defined forms, compare with
  Sections 3.1.6 and 3.1.7 of \cite{book}.
\end{remark}

\begin{example}\label{ex2.1} 
  (1) Being a filtered manifold which is regular of type $\frak m$ may
  already be an interesting geometric structure in its own right. So
  we can take the case $G_0=\Aut_{\gr}(\frak m)$ and $\Cal G_0=\Cal
  PM$ as a filtered $G_0$--structure. It is known from examples of
  parabolic geometries that such structures may already be of finite
  type and determine canonical Cartan geometries, see Section 4.3.1 of
  \cite{book}.

  Let us in particular mention that a standard way to get to filtered
  manifolds is to start from a bracket--generating distribution
  $H=T^{-1}M\subset TM$ on a smooth manifold $M$. Then one considers
  the subspaces in the tangent spaces spanned by sections of $H$ and
  by brackets of two such sections. Assuming that these subspaces all
  have the same dimension, they define a smooth subbundle
  $T^{-2}M\supset T^{-1}M$. Proceeding in that way, one obtains a
  filtered manifold (if the constant rank assumption is satisfied in
  each step). As a further regularity assumption on $H$, one can then
  require that the resulting filtered manifold is regular of type
  $\frak m$ for a nilpotent graded Lie algebra $\frak m$ which is
  automatically fundamental.

  In several cases, all these regularity properties are consequences
  of some genericity assumptions on $H$. For example if $\dim(M)=6$
  and $H$ has rank $3$, then one may assume that sections of $H$
  together with Lie brackets of two such sections span the full
  tangent space in each point. This automatically implies that $M$ is
  regular of type $\frak m=\frak m_{-2}\oplus\frak m_{-1}$ with $\frak
  m_{-1}=\Bbb R^3$ and $\frak m_{-2}=\La^2\Bbb R^3$ with the wedge
  product as the Lie bracket. These are the distributions studied by
  R.~Bryant in his thesis, see \cite{Bryant}. Similarly, generic rank
  two distributions in dimension five as studied in E.~Cartan's ``five
  variables paper'' \cite{Cartan:five} automatically give rise to
  regular filtered manifolds with $\frak m$ the free three--step
  nilpotent Lie algebra on two generators.

  (2) The analogy to classical G--structures has to be taken with a
  bit of care. Not all structures that look like filtered analogs of
  G--structures actually are filtered $G_0$--structures. Let us
  consider the example of contact manifolds, which by definition are
  just filtered manifolds that are regular of type $\frak m$ for a
  Heisenberg algebra $\frak m$. This means that $\frak m=\frak
  m_{-2}\oplus\frak m_{-1}$ with $\frak m_{-2}\cong\Bbb R$ and such
  that the bracket $\frak m_{-1}\x\frak m_{-1}\to\frak m_{-2}$ is
  non--degenerate as a bilinear map (which implies that $\frak m_{-1}$
  has even dimension). Fixing an identification of $\frak m_{-2}$ with
  $\Bbb R$, the bracket defines a symplectic form on $\frak m_{-1}$
  and $\Aut_{\gr}(\frak m)$ is isomorphic to the conformally
  symplectic group $CSp(\frak m_{-1})\subset GL(\frak m_{-1})$.

  Recall that a sub--Riemannian metric on a contact manifold $(M,H)$
  is given by a smooth family $g_x$ of inner products on the spaces
  $H_x$ for $x\in M$. Unfortunately, a sub--Riemannian metric is not a
  filtered $G_0$--structure in general. The problem is that already in
  the model case, there is not only one positive definite inner
  product on $\frak m_{-1}$ up to equivalence. Given an inner product
  $\langle\ ,\ \rangle$ on $\frak m_{-1}$ we can identify $\frak
  m_{-2}$ with $\Bbb R$ and then diagonalize the skew symmetric
  bilinear form defined by the bracket with respect to the inner
  product. If $\dim(\frak m_{-1})=2k$, then this gives $k$ eigenvalues
  determined up to sign, which only change by an overall factor upon
  changing the identification of $\frak m_{-2}$ with $\Bbb R$. Thus
  the ratios of the positive eigenvalues are independent of all
  choices. Clearly, if two inner products on $\frak m_{-1}$ are
  equivalent under the action of $CSp(\frak m_{-1})$, they must lead to
  the same ratios of eigenvalues. So at least if $\dim(\frak
  m_{-1})\geq 4$, the isomorphism classes of inner products on $\frak
  m_{-1}$ depend on continuous parameters. Different values of these
  parameters may lead to non--isomorphic stabilizers of the inner
  product within $CSp(\frak m_{-1})$ and even to stabilizers of
  different dimension.

  To obtain a filtered $G_0$--structure there have to be isomorphisms
  $H_x\to\frak m_{-1}$ for all points $x\in M$, which at the same time 
  are compatible with the conformally symplectic structures on both spaces
  and with $g_x$ and a fixed inner product on $\frak
  m_{-1}$. This clearly shows that for the continuous invariants for
  the inner products $g_x$ all have to be constant in order for a
  sub--Riemannian metric to define a filtered $G_0$--structure, which
  is a very restrictive condition. If this condition is satisfied,
  however, then sub--Riemannian metrics nicely fit into the general
  concept of filtered $G_0$--structures.
\end{example}

\subsection{Admissible pairs}\label{2.2}
If it is possible to associate a canonical Cartan connection to
filtered G--structures of some fixed type, then there must be a
homogeneous model for the geometry (at least on an infinitesimal
level). It is rather easy to describe existence of a homogeneous
filtered $G_0$--structure infinitesimally.

Consider a Lie group $G$ and a closed subgroup $P\subset G$ and let
$\frak p\subset\frak g$ be their Lie algebras. We make the standard
assumption that the action of $G$ on $G/P$ is infinitesimally
effective, see Section 1.4.1 in \cite{book}. This means that any
normal subgroup of $G$ contained in $P$ has to be discrete, or
equivalently, that there is no non--trivial ideal of $\frak g$ which
is contained in $\frak p$. Observe that for a normal subgroup
$K\subset G$ that is contained in $P$, one may always replace $(G,P)$
by $(G/K,P/K)$, which leads to the same homogeneous space. Allowing
non--trivial discrete subgroups $K$ is again done to include
structures like spin structures.

Then it is well known that for the homogeneous space $G/P$, the
tangent bundle is the associated bundle $T(G/P)\cong G\x_P(\frak
g/\frak p)$. Hence a $G$--invariant filtration
$\{T^{i}(G/P)\}_{i=-1}^{-\mu}$ is equivalent to a sequence of
$P$--invariant subspaces in $\frak g/\frak p$. Taking the pre--images
in $\frak g$ we get a sequence
$$
\frak g=\frak g^{-\mu}\supset \frak g^{-\mu+1}\supset\dots\supset\frak
g^{-1}\supset\frak g^0,
$$ 
of $\Ad(P)$--invariant subspaces, where we put $\frak g^0:=\frak
p$. The corresponding subbundles in $T(G/P)$ then are the images of
the subbundles of $TG$ spanned by left--invariant vector fields with
generators in these subspaces. This readily implies that the
filtration $\{T^{i}(G/P)\}$ makes $G/P$ into a filtered manifold if
and only if $[\frak g^i,\frak g^j]\subset\frak g^{i+j}$ for all
$i,j\leq 0$. 

There is an obvious way to continue this filtration. Define $\frak
g^1\subset\frak g^0$ as the space of those elements $X\in\frak g^0$
such that $\ad(X)(\frak g^i)\subset\frak g^{i+1}$ for all
$i=-\mu,\dots,-1$. The Jacobi identity then implies that we also have
$[\frak g^0,\frak g^1]\subset\frak g^1$. Then define $\frak
g^2\subset\frak g^1$ as the space of those elements $X$ for which
$\ad(X)(\frak g^i)\subset\frak g^{i+2}$ for all
$i=-\mu,\dots,-1$. Again by the Jacobi identity, $[\frak g^0,\frak
g^2]$ and $[\frak g^1,\frak g^1]$ are both contained in $\frak g^2$.
Inductively we obtain a sequence of subspaces $\frak g^j\subset\frak
g^0$ for all $j>0$ such that $\frak g^{j+1}\subset\frak g^j$ for all
$j$ and such that $[\frak g^i,\frak g^j]\subset\frak g^{i+j}$ provided
that all three spaces have been defined already. Since $\frak g$ is
finite dimensional, this sequence of subspaces has to stabilize at
some stage, and we denote by $\nu$ the largest index such that $\frak
g^\nu\neq \frak g^{\nu+1}$. By construction, we then obtain that
$[\frak g^{\nu+1},\frak g]\subset\frak g^{\nu+1}$. Thus $\frak
g^{\nu+1}$ is an ideal in $\frak g$ that is contained in $\frak g^0$, so
$\frak g^{\nu+1}=\{0\}$ by infinitesimal effectivity. Hence we
conclude that we get a filtration of $\frak g$ of the form
\begin{equation}\label{filt-def}
  \frak g=\frak g^{-\mu}\supset\dots\supset\frak g^{-1}\supset\frak
  g^0=\frak p\supset\frak g^1\supset\dots\supset\frak g^\nu.
\end{equation}
This makes $\frak g$ into a \textit{filtered Lie algebra} in the sense
that $[\frak g^i,\frak g^j]\subset\frak g^{i+j}$ for all $i,j$, where
we agree that $\frak g^\ell=\frak g$ for $\ell<-\mu$ and $\frak
g^\ell=\{0\}$ for $\ell>\nu$. Consider the automorphism group
$\Aut(\frak g)$ of $\frak g$, which is a closed subgroup of $GL(\frak
g)$ and thus a Lie group. Define $GL_f(\frak g):=\{\ph\in GL(\frak
g):\forall i:\ph(\frak g^i)\subset\frak g^i\}$, the subgroup of
elements of $GL(\frak g)$ which preserve the filtration of $\frak g$,
and put $\Aut_f(\frak g)=\Aut(\frak g)\cap GL_f(\frak g)$. Then
$GL_f(\frak g)$ and $\Aut_f(\frak g)$ are closed subgroups and thus
Lie subgroups of $GL(\frak g)$. Their Lie algebras are the spaces
$L_f(\frak g,\frak g)$ and $\frak{der}_f(\frak g)$ of filtration
preserving linear maps and derivations, respectively. Abstracting the
properties derived here motivates the following definition.

\begin{definition}\label{def2.2}
  An \textit{admissible pair} $(\frak g,P)$ consists of
\begin{itemize}
\item[(i)] A Lie algebra $\frak g$ endowed with a filtration $\{\frak
  g^i\}_{i=-\mu}^\nu$ as in \eqref{filt-def} making $\frak g$ into a
  filtered Lie algebra.
\item[(ii)] A Lie group $P$ with Lie algebra $\frak p:=\frak g^0$. 
\item[(iii)] A group homomorphism $\Ad:P\to\Aut_f(\frak g)$ whose
  derivative coincides with $\ad|_{\frak g^0}:\frak
  g^0\to\frak{der}_f(\frak g)$. 
\end{itemize}  
such that 
\begin{itemize}
\item [(A)] There is no ideal in $\frak g$ which is contained in
  $\frak g^0$ (``infinitesimal effectivity'').
\item [(B)] If $A\in\frak g^0$ is such that for all $i=-\mu,\dots,-1$,
  we have $\ad(A)(\frak g^i)\subset\frak g^{i+1}$, then $A\in\frak
  g^1$. 
\end{itemize}
\end{definition}

In the above considerations, we have only assumed that we have given
$\frak g^i\subset\frak g$ for $i\leq 0$, and then constructed specific
filtration components for $i>0$. In the definition of an admissible
pair, we start with an arbitrary filtration, and condition (B) just
ensures that we get the same subspace $\frak g^1\subset\frak g^0$ as
constructed above. We will see later that the fact that we get the
``right'' higher filtration components is a consequence of the an
assumption on prolongations that we will impose later on.

Observe that for an admissible pair $(\frak g,P)$ and a subgroup $Q\subset
P$ which contains the connected component of the identity of $P$, also
$(\frak g,Q)$ is an admissible pair.

\begin{example}\label{ex2.2}
  Suppose that $(\frak g,\{\frak g^i\}_{i=-\mu}^\nu)$ is a filtered Lie
  algebra which satisfies the conditions (A) and (B) from Definition
  \ref{def2.2}. Suppose further that $\frak g^0\subset\frak g$
  coincides with its normalizer in $\frak g$, so if $X\in\frak g$
  satisfies $[X,\frak g^0]\subset\frak g^0$ then $X\in \frak g^0$.
  Then given a Lie group $G$ with Lie algebra $\frak g$, define 
$$
P:=\{g\in G:\Ad(g)\in GL_f(\frak g)\}\subset G.
$$ 
This is the pre--image of the closed subgroup $GL_f(\frak g)\subset
GL(\frak g)$ under a smooth homomorphism and thus a closed subgroup,
too. The Lie algebra of $P$ by construction is $\{X\in\frak
g:\ad(X)\in L_f(\frak g,\frak g)\}$, so by definition it contains
$\frak g^0$. On the other hand, the condition on the normalizer
implies that $\ad(X)(\frak g^0)\subset\frak g^0$ already implies
$X\in\frak g^0$, so $P$ has Lie algebra $\frak g^0$. Thus we see that
restricting the adjoint action of $G$ to $P$ makes $(\frak g,P)$ into
an admissible pair.

Without the assumption on the normalizer, one can take any closed
subgroup $\tilde P\subset G$ with Lie algebra $\frak g^0$ and then
form the closed subgroup $P:=\{g\in P:\Ad(g)\in GL_f(\frak g)\}$. As
above, one concludes that this contains the connected component of the
identity of $\tilde P$ and hence has Lie algebra $\frak g^0$ and that
$(\frak g,P)$ is an admissible pair.
  \end{example}

\subsection{Passing to the associated graded}\label{2.3}
For a filtered Lie algebra $(\frak g,\{\frak g^i\})$, one can form the
associated graded vector space to $\frak g$, which inherits a
canonical Lie algebra structure. We put $\gr_i(\frak g):=\frak
g^i/\frak g^{i+1}$ and then define $\gr(\frak
g):=\oplus_{i=-\mu}^\nu\gr_i(\frak g)$. Then we observe that $[X+\frak
  g^{i+1},Y+\frak g^{j+1}]:=[X,Y]+\frak g^{i+j+1}$ is a well--defined
bilinear map $\gr_i(\frak g)\x\gr_j(\frak g)\to\gr_{i+j}(\frak
g)$. Putting these maps together we obtain a bracket $[\ ,\ ]$ on
$\gr(\frak g)$, which makes that space into a \textit{graded Lie
  algebra}. Observe that there is neither a canonical map from $\frak
g$ to $\gr(\frak g)$ nor in the opposite direction and that $\frak g$
and $\gr(\frak g)$ are not isomorphic Lie algebras in general.

The action of $P$ on $\frak g$ by definition preserves the filtration
$\{\frak g^i\}$. Hence for $g\in P$ and each $i$, the linear
isomorphism $\Ad(g):\frak g^i\to\frak g^i$ descends to an isomorphism
$\gr_i(\frak g)\to\gr_i(\frak g)$. Taking these maps together, we
obtain a linear isomorphism $\Adgr(g):\gr(\frak g)\to\gr(\frak g)$
compatible with the grading. Since $\Ad(g)$ is a Lie algebra
automorphism on $\frak g$, we easily conclude that $\Adgr(g)$ is an
automorphism of the graded Lie algebra $\gr(\frak g)$. Of course, this
defines a smooth homomorphism $\Adgr:P\to \Aut_{\gr}(\gr(\frak g))$.

Next, consider the negative part $\frak
m:=\oplus_{i=-\mu}^{-1}\gr_i(\frak g)$ of $\frak g$. By the grading
property, this is a nilpotent graded subalgebra of $\gr(\frak
g)$. Hence we can restrict automorphisms and derivations of $\gr(\frak
g)$ which preserve the grading to the subalgebra $\frak m$, thus
obtaining homomorphisms $\Aut_{\gr}(\gr(\frak g))\to\Aut_{\gr}(\frak
m)$ and $\frak{der}_{\gr}(\gr(\frak g))\to\frak{der}_{\gr}(\frak
m)$. In general, these homomorphisms are neither injective nor
surjective, but under the assumptions we have imposed, we can prove
the following. 

\begin{prop}\label{prop2.3}
  The kernel of the homomorphism $\Adgr:P\to\Aut_{\gr}(\gr(\frak g))$
  is a closed normal subgroup $P_+\subset P$ with Lie algebra $\frak
  g^1$. Denoting the quotient group $P/P_+$ by $G_0$, the Lie algebra
  of $G_0$ can be naturally identified with $\gr_0(\frak g)$ and the
  homomorphism $\Adgr$ descends to an infinitesimally injective
  homomorphism $G_0\to\Aut_{\gr}(\gr(\frak g))$. Via restriction to
  $\frak m\subset\gr(\frak g)$, one obtains a homomorphism
  $G_0\to\Aut_{\gr}(\frak m)$, which is also infinitesimally
  injective.
\end{prop}
\begin{proof}
  Since $\Adgr$ is a smooth homomorphism of Lie groups, its kernel is
  a closed normal subgroup $P_+\subset P$. Next, denote by
  $GL_{\gr}(\gr(\frak g))$ the group of all linear automorphisms of
  $\gr(\frak g)$ that preserve the grading. Then there is an obvious
  homomorphism $GL_f(\frak g)\to GL_{\gr}(\gr(\frak g))$ such that the
  image of $\ph\in GL_f(\frak g)$ is given on $\gr_i(\frak g)$ as the
  isomorphism induced by $\ph|_{\frak g^i}:\frak g^i\to\frak g^i$. On
  the Lie algebra level, this corresponds to the map $L_f(\frak
  g,\frak g)\to L_{\gr}(\gr(\frak g),\gr(\frak g))$ obtained in the
  same way.

  By construction, $\Adgr$ is simply the composition of $\Ad$ with
  this homomorphism, so the derivative of $\Adgr$ maps $X\in\frak g^0$
  to the map $\gr(\frak g)\to\gr(\frak g)$ induced by $\ad(X)\in
  L_f(\frak g,\frak g)$. Now the Lie algebra of $P_+$ by construction
  coincides with the kernel of this derivative. Hence it consists of
  all $X$ such that $\ad(X)(\frak g^i)\subset\frak g^{i+1}$ for all
  $i=-\mu,\dots,\nu$. This evidently contains $\frak g^1$ and by
  condition (B) in Definition \ref{def2.2} it actually coincides with
  $\frak g^1$.

  Now it is clear by construction that $\Adgr$ descends to an
  infinitesimally injective homomorphism
  $P/P_+=G_0\to\Aut_{\gr}(\gr(\frak g))$. Restricting the resulting
  maps, we get a homomorphism $G_0\to\Aut_{\gr}(\frak m)$. But by
  construction the kernel of the derivative of the composition $P\to
  G_0\to \Aut_{\gr}(\frak m)$ consists of all $X\in\frak g^0$ such
  that $\ad(X)(\frak g^i)\subset\frak g^{i+1}$ for all
  $i=-\mu,\dots,-1$. Again by condition (B) in Definition
  \ref{def2.2}, this coincides with $\frak g^1$, so the last claim
  follows.
\end{proof}

\begin{example}\label{ex2.3}
  A simple but important example showing that different filtered Lie
  algebras may lead to the same data on the level of the associated
  graded is related to model mutation, see Definition 3.8 in
  \cite{Sharpe}. Consider the group $G=O(n+1,\Bbb R)$ and let
  $P:=O(n)\subset G$ be the stabilizer of the hyperplane $\Bbb
  R^n\subset\Bbb R^{n+1}$. Then on the Lie algebra $\frak
  g=\frak{o}(n+1)$, we define a filtration by $\frak g^{-1}=\frak g$
  and $\frak g^0:=\frak{o}(n)=\frak p\subset\frak g$. (This makes
  $\frak g$ into a filtered Lie algebra, since $\frak g^0$ is a Lie
  subalgebra of $\frak g$, so any homogeneous space corresponds to an
  admissible pair.) In matrix form, we can decompose any skew
  symmetric matrix into blocks of size $n$ and $1$ as $\begin{pmatrix}
    A & v\\ -v^t & 0\end{pmatrix}$ with $A\in\frak{o}(n)$ and
  $v\in\Bbb R^n$. Denoting elements of $\frak g$ by $(A,v)$ we see
  that $\frak g^0$ corresponds to the elements of the form
  $(A,0)$. Hence $\gr(\frak g)=\gr_{-1}(\frak g)\oplus\gr_0(\frak
  g)\cong\Bbb R^n\oplus\frak{o}(n)$, and the only non--zero brackets
  are the bracket on $\gr_0(\frak g)\cong\frak g^0$ and the one
  $\gr_0(\frak g)\x\gr_{-1}(\frak g)\to\gr_{-1}(\frak g)$ given by the
  standard action of $\frak{o}(n)$ on $\Bbb R^n$. So the fact that two
  matrices $(0,v)$ and $(0,w)$ in general have non-trivial bracket
  (contained in $\frak g^0$) is forgotten when passing to the
  associated graded. Likewise, the action of $P=O(n)$ on $\gr(\frak
  g)$ is just the direct sum of the standard action on $\Bbb R^n$ and
  the adjoint action on $\frak{o}(n)$.

  The key issue about this example is that one can start in a very
  similar way starting with $G=\text{Euc}(n)$, the group of Euclidean
  motions and $P\cong O(n)\subset G$ the stabilizer of a point, or
  with $G=O(n,1)$ and $P\cong O(n)\subset G$ the stabilizer of a
  positive hyperplane. Both these examples lead to the same graded Lie
  algebra $\gr(\frak g)$, the same group $P$, and the same action of
  $P$ on $\gr(\frak g)$. This corresponds to the fact that one may
  equally well take Euclidean space, the sphere, or hyperbolic space
  as the homogeneous model of Riemannian geometry. The difference
  between the three models only shows up in the resulting notion of
  curvature. Here Euclidean space leads to the standard notion of
  Riemann curvature while the other models lead to a shift by a
  curvature tensor of constant sectional curvature chosen in such a
  way that the sphere respectively hyperbolic space have zero
  curvature.
\end{example}

\subsection{Cartan geometries and underlying structures}\label{2.4} 
Traditionally, Cartan geometries are defined starting from a pair
$(G,P)$, but it is clear that there is no problem to start from a pair
$(\frak g,P)$ instead. So given an admissible pair $(\frak g,P)$, and
a smooth manifold $M$ a \textit{Cartan geometry} $(p:\Cal G\to M,\om)$
of type $(\frak g,P)$ on $M$ is given by a principal fiber bundle
$p:\Cal G\to M$ with structure group $P$, which is endowed with a
\textit{Cartan connection} $\om\in\Om^1(\Cal G,\frak g)$. Denoting by
$r^g$ the principal right action of an element $g\in P$ and by $\ze_X$
the fundamental vector field generated by and element $X\in\frak p$,
the defining properties of a Cartan connection are
$(r^g)^*\om=\Ad(g^{-1})\o\om$, $\om(\ze_X)=X$, and the fact that for
each point $u\in\Cal G$, the value $\om_u:T_u\Cal G\to\frak g$ is a
linear isomorphism. Observe that the last condition forces the
dimension of $M$ to be equal to $\dim(\frak g)-\dim(\frak
g^0)$. Together with the condition on fundamental vector fields, we
also conclude that the vertical subbundle of $p:\Cal G\to M$ can be
characterized via $V_uP=\{\xi\in T_uP:\om_u(\xi)\in\frak g^0\}\subset
T_uP$.

The \textit{curvature} of a Cartan connection $\om\in\Om^1(\Cal
G,\frak g)$ is the two--form $K\in\Om^2(\Cal G,\frak g)$ defined by
$K(\xi,\eta):=d\om(\xi,\eta)+[\om(\xi),\om(\eta)]$ for
$\xi,\eta\in\frak X(\Cal G)$. It easily follows from the defining
properties of a Cartan connection $\om$ that $K$ is equivariant for
the principal right action and \textit{horizontal},
i.e.~$(r^g)^*K=\Ad(g^{-1})\o K$ for all $g\in P$ and $0=K(\ze_X,\eta)$
for any $X\in\frak g^0=\frak p$ and $\eta\in\frak X(\Cal G)$. This is
also a consequence of Proposition \ref{prop4.1}, whose proof is
independent of what we are doing here. 

A classical concept is that the Cartan connection $\om$ is called
\textit{torsion--free} if and only if $K(\xi,\eta)\in\frak
g^0\subset\frak g$ for all tangent vectors $\xi$ and $\eta$ on $\Cal
G$. A weakening that will be crucial for the further development is the
concept of regularity. We call the Cartan connection $\om$
\textit{regular} if for tangent vectors $\xi,\eta\in T_u\Cal G$ such
that $\om_u(\xi)\in\frak g^i$ and $\om_u(\eta)\in\frak g^j$ for some
$i,j$, we always have $K_u(\xi,\eta)\in\frak g^{i+j+1}$. Observe that
this condition is always satisfied if one of the indices is $\geq 0$
by horizontality of $K$. Since for negative indices $i$ and $j$, we
always have $i+j+1<0$, we conclude that a torsion--free Cartan
connection is automatically regular. 

Now we can prove that any regular Cartan geometry modelled on an
admissible pair gives rise to an underlying filtered
$G_0$--structure. Suppose that we have a filtered Lie algebra $(\frak
g,\{\frak g^i\}_{i=-\mu}^\nu)$ with associated graded $\gr(\frak
g)=\oplus_{i=-\mu}^\nu\gr_i(\frak g)$. In Section \ref{2.3} above, we
have observed that the negative part $\mathfrak
m:=\oplus_{i=-\mu}^{-1}$ is nilpotent graded Lie subalgebra of
$\gr(\frak g)$. For $A\in\gr_0(\frak g)$, we can restrict the adjoint
action of $A$ on $\gr(\frak g)$ to $\frak m$, thus obtaining a
homomorphism $\ad_{\frak m}:\frak g_0\to\frak{der}_{\gr}(\frak
m)$. Observe that this is the derivative of the homomorphism
$G_0\to\Aut_{\gr}(\frak m)$ constructed in Proposition \ref{prop2.3},
so we have seen there that $\ad_{\frak m}$ is injective. Observe that
this proposition shows that the concept of a filtered $G_0$--structure
makes sense on regular filtered manifolds of type $\frak m$.

\begin{thm}\label{thm2.4}
  Let $(\frak g,P)$ be an admissible pair, $\gr(\frak g)$ the
  associated graded Lie algebra to $\frak g$, and $\frak m$ its
  negative part. Let $P_+\subset P$ be the subgroup defined in
  Proposition \ref{prop2.3} and put $G_0:=P/P_+$.

  Then any regular Cartan geometry $(p:\Cal G\to M,\om)$ of type
  $(\frak g,P)$ over a smooth manifold $M$ naturally induces a
  filtration $\{T^iM\}_{i=-\mu}^{-1}$ of the tangent bundle $TM$, which
  makes $M$ into a filtered manifold that is regular of type
  $\mathfrak m$, as well as a filtered $G_0$--structure on $M$.
\end{thm}
\begin{proof}
  By definition, for each $u\in\Cal G$ the map $\om_u:T_u\Cal
  G\to\frak g$ is a linear isomorphism. Thus for $i=-\mu,\dots,\nu$, we
  can define $T^i_u\Cal G\subset T_u\Cal G$ as the subspace consisting
  of all tangent vectors $\xi$ such that $\om_u(\xi)\in\frak
  g^i\subset\frak g$. Smoothness of $\om$ immediately implies that
  these spaces fit together to define smooth subbundles $T^i\Cal
  G\subset T\Cal G$ such that $T^i\Cal G\supset T^{i+1}\Cal G$ for all
  $i$. Moreover, by definition, $T^0\Cal G$ is the vertical subbundle
  $\ker(Tp)$ of $p:\Cal G\to M$. In particular, for each $u\in\Cal G$,
  the map $\om_u:T_u\Cal G\to\frak g$ descends to a linear isomorphism
  $T_u\Cal G/T^0_u\Cal G\to\frak g/\frak g^0$, and since $T^0_u\Cal
  G=\ker(T_up)$, the left hand space is naturally isomorphic to
  $\im(T_up)=T_{p(u)}M$.

  Equivariancy of $\om$ shows that for $\xi\in T_u^i\Cal G$ and $g\in
  P$ with principal right action $r^g:\Cal G\to\Cal G$, we get
$$
\om_{u\cdot
  g}(T_ur^g\cdot\xi)=((r^g)^*\om)(u)(\xi)=\Ad(g^{-1})(\om_u(\xi))\in \frak
g^i. 
$$ Since $\Ad(g^{-1})\in\Aut_f(\frak g)$, the subbundle $T^i\Cal G$ is
invariant under $Tr^g$ for each $i$ and each $g\in P$. Now for a point
$x\in M$ we can choose a point $u\in\Cal G$ such that $p(u)=x$ and
consider the linear isomorphism $\ph_u:T_xM\to\frak g/\frak g^0$ from
above. Any other point over $x$ is of the form $u\cdot g$ for some
element $g\in P$ and we conclude that $\ph_{u\cdot
  g}=\Adb(g^{-1})\o\ph_u$. Here $\Adb(g^{-1})$ denotes the linear
automorphism of $\frak g/\frak g^0$ induced by $\Ad(g^{-1})$. In
particular, for all $i=-\mu,\dots,-1$, the pre--image
$\ph_u^{-1}(\frak g^i/\frak g^0)\subset T_xM$ is independent of the
choice of $u$, thus giving rise to a well defined linear subspace
$T^i_xM\subset T_xM$.

Now take a local smooth section $\si:U\to\Cal G$ of the principal bundle
$p:\Cal G\to M$, let $\pi:\frak g\to\frak g/\frak g^0$ be the
canonical projection, and consider $\pi\o\si^*\om\in \Om^1(M,\frak
g/\frak g^0)$. By construction, for each $x\in U$, this restricts to
the linear isomorphism $\ph_{\si(x)}:T_xM\to \frak g/\frak g^0$. Hence
it defines a trivialization of $TM|_U$ under which the subspace
$T^i_xM$ for $x\in U$ correspond to $\frak g^i/\frak g^0\subset\frak
g/\frak g^0$. Thus we see that we have actually constructed smooth
subbundles $TM=T^{-\mu}M\supset\dots\supset T^{-1}M$, and we claim
that these make $M$ into a filtered manifold which is regular of type
$\frak m$. 

To see this, take local smooth sections $\xi\in\Ga(T^iM)$ and $\eta\in
\Ga(T^jM)$, defined on a subset of $M$ over which $\Cal G$ is trivial. Then
there are smooth lifts $\tilde\xi,\tilde\eta\in\frak X(\Cal G)$ and by
construction we have $\tilde\xi\in\Ga(T^i\Cal G)$ and
$\tilde\eta\in\Ga(T^j\Cal G)$. It is a basic fact of differential
geometry that $[\tilde\xi,\tilde\eta]$ then is a lift of
$[\xi,\eta]\in\frak X(M)$. Using the definition of the exterior
derivative and of the curvature $K$, we now compute
\begin{equation}\label{bracket}
\begin{aligned}
  \om\left(\left[\tilde\xi,\tilde\eta\right]\right)&=
  -d\om(\tilde\xi,\tilde\eta)+
  \tilde\xi\cdot\om(\tilde\eta)-\tilde\eta\cdot\om(\tilde\xi)\\
  &= -K(\tilde\xi,\tilde\eta)+[\om(\tilde\xi),\om(\tilde\eta)]+
  \tilde\xi\cdot\om(\tilde\eta)-\tilde\eta\cdot\om(\tilde\xi).
\end{aligned}
\end{equation}
Now by assumption, the function $\om(\tilde\xi)$ has values in $\frak
g^i\subset\frak g^{i+j+1}$, so the same holds for the derivative
$\tilde\eta\cdot\om(\tilde\xi)$. Likewise,
$\tilde\xi\cdot\om(\tilde\eta)$ has values in $\frak g^j\subset\frak
g^{i+j+1}$, and by regularity, also $K(\tilde\xi,\tilde\eta)$ has
values in $\frak g^{i+j+1}$. Finally,
$[\om(\tilde\xi),\om(\tilde\eta)]$ has values in $[\frak g^i,\frak
g^j]\subset\frak g^{i+j}$, which shows that $[\xi,\eta]$ is a section
of $T^{i+j}M$, so $(M,\{T^iM\})$ is a filtered manifold. 

Now of course, the above local simultaneous trivializations of the
bundles $T^iM$ induces local trivializations $\gr_i(TM)|_U\cong
U\x\gr_i(\frak g)$, so one obtains a local trivialization
$\gr(TM)|_U\cong U\x\frak m$. By construction, this can be explicitly
described as follows: Given $v\in\Ga(\gr_i(TM)|_U)$, first choose a
representative vector field $\xi\in\Ga(T^iU)$ and then $v$ corresponds
to the function $U\to\frak g^i/\frak g^{i+1}$ defined by
$\si^*\om(\xi)+\frak g^{i+1}$. Of course, the same result is obtained
if one applies $\om$ to any other lift of $\xi$ in $T_{\si(x)}\Cal
G$. But now in the above situation, the class of
$[\om(\tilde\xi),\om(\tilde\eta)]$ in $\frak g^{i+j}/\frak g^{i+j+1}$
coincides with the bracket in $\frak m\subset\gr(\frak g)$ of the
elements $\om(\tilde\xi)+\frak g^{i+1}$ and $\om(\tilde\eta)+\frak
g^{j+1}$. By the above argument, this coincides with the class of
$\om([\tilde\xi,\tilde\eta])$, thus representing the class of
$[\xi,\eta]$ in $\gr_{i+j}(TM)$. This shows that in our local
trivialization the Levi--bracket on $M$ is represented by the Lie
bracket on $\frak m$, which shows that the filtered manifold
$(M,\{T^iM\})$ is regular of type $\frak m$.

To construct the filtered $G_0$--structure, observe that the closed
normal subgroup $P_+\subset P$ acts freely on $\Cal G$ by the
restriction of the principal right action. Using a local
trivialization of $\Cal G$, one easily concludes that the orbit space
$\Cal G_0:=\Cal G/P_+$ endowed with the obvious projection $p_0:\Cal
G_0\to M$ is a principal fiber bundle with structure group
$P/P_+=G_0$. As we have observed above, for a point $u\in\Cal G$ we
obtain an isomorphism $\ph_u:T_{p(u)}M\to\frak g/\frak g^0$ which is
compatible with the filtrations on the two spaces. Hence we can pass
to the induced linear isomorphism
$\underline{\ph}_u:\gr(T_{p(u)}M)\to\frak m=\gr(\frak g/\frak
g^0)$. We have seen already that $\ph_{u\cdot g}=\Adb(g^{-1})\o\ph_u$,
which readily implies that $\underline{\ph}_{u\cdot
  g}=\Adgr(g^{-1})|_{\frak m}\o\underline{\ph}_u$. But now by
definition $g\in P_+$ implies that $\Adgr(g^{-1})=\id$ so
$\underline{\ph}_u$ depends only on the class of $uP_+\in\Cal
G/P_+=\Cal G_0$. Hence for any point $u_0\in\Cal G_0$, we obtain a
linear isomorphism $\ps_{u_0}:\gr(T_{p_0(u_0)}M)\to\frak m$ and for
$g_0\in G_0$, we get $\ps_{u_0\cdot g_0}=\Adgr(g^{-1})\o\ps_{u_0}$,
where $g\in P$ is any element such that $gP_+=g_0$. Using a local
trivialization of $\Cal G$ and the induces local trivialization of
$\Cal G_0$, one easily shows that this depends smoothly on $u_0$, thus
defining a reduction of structure group $\Cal G_0\to\Cal PM$ as
required.
\end{proof}

Take an admissible pair $(\frak g,P)$ and suppose that $G$ is a Lie
group with Lie algebra $\frak g$ that contains $P$ as a closed
subgroup. Then of course $G\to G/P$ is a principal $P$--bundle and the
left Maurer--Cartan form makes this into a Cartan geometry, which is
flat by the Maurer--Cartan equation. Hence it is regular so by the
theorem, $G/P$ is a regular filtered manifold of type $\frak m$ and
$G/P_+\to G/P$ is a filtered $G_0$--structure. By construction all
these structures are homogeneous under the action of $G$, so we have
found many examples of homogeneous filtered $G_0$--structures.

\subsection{Tanaka prolongation}\label{2.5} 
At the current stage, we have just encoded the fact that a Cartan
geometry induces an underlying filtered $G_0$--structure into
algebraic data for the modelling pair $(\frak g,P)$. This by no means
implies that a Cartan geometry of this type should be canonically
associated to this underlying filtered $G_0$--structure. We simply
have not imposed any conditions in that direction so far. An algebraic
condition serving that purpose has been introduced in the pioneering
work of N.~Tanaka, see \cite{Tanaka:prolon}. Since it is actually
phrased in the language of graded Lie algebras, we can directly impose
this condition in our setting.

\begin{definition}\label{def2.5}
  Let $(\frak g,P)$ be an admissible pair, let $\{\frak
  g^i\}_{i=-\mu}^\nu$ be the corresponding filtration of $\frak g$,
  $\gr(\frak g)$ the associated graded Lie algebra and $\frak m$ its
  negative part.

  (1) We say that $\gr(\frak g)$ is the \textit{full prolongation of
    $(\frak m,\gr_0(\frak g))$} if for each $j\geq 1$ and each linear
  map $\ph:\frak m\to\gr(\frak g)$ that is homogeneous of degree $j$
  (i.e.~satisfies $\ph(\frak m_i)\subset\gr_{i+j}(\frak g)$) such that
  for all $X,Y\in\frak m$ we have $\ph([X,Y])=[\ph(X),Y]+[X,\ph(Y)]$
  there is an element $Z\in\gr_j(\frak g)$ such that
  $\ph=\ad(Z)|_{\frak m}:\frak m\to\gr(\frak g)$.

In this case, we say that $(\frak g,P)$ is an \textit{infinitesimal
  homogeneous model} for filtered $G_0$--structures, where
$G_0=P/P_+$.

  (2) We say that $\gr(\frak g)$ is the \textit{full prolongation of
    $\frak m$} if in addition $\ad_{\frak m}:\gr_0(\frak g)\to
  \frak{der}_{\gr}(\frak m)$ is an isomorphism. 

  In this case, we say that $(\frak g,P)$ is an \textit{infinitesimal
    homogeneous model} for filtered manifolds that are regular of type
  $\frak m$.
\end{definition}

In our context, this condition is quite easy to understand. Observe
that for each $i>0$ and each $Z\in\gr_i(\frak g)$, the Jacobi identity
for $\frak g$ shows that the map $\ad(Z)|_{\frak m}:\frak
m\to\gr(\frak g)$ satisfies
$\ad(Z)([X,Y])=[\ad(Z)(X),Y]+[X,\ad(Z)(Y)]$. This works for any
admissible pair inducing a filtered $G_0$--structure. So the condition
of being the full prolongation says that $(\frak g,P)$ is (in some
sense) maximal among the admissible pairs inducing a filtered
$G_0$--structure. For (finite dimensional) Cartan geometries of type
$(\frak g,P)$ being equivalent to the underlying filtered
$G_0$--structure (i.e.~not encoding additional data) should certainly
imply that $(\frak g,P)$ is maximal in this sense.

\begin{remark}\label{rem2.5} 
  It turns out that if $(\frak g,P)$ is an admissible pair such that
  $\gr(\frak g)$ is the full prolongation of $(\frak m,\gr_0(\frak
  g))$, then the filtration on $\frak g$ is obtained from its
  non--positive part as described in Section \ref{2.2}. This means
  that if $A\in\frak g^1$ has the property that $\ad(A)(\frak
  g^i)\subset\frak g^{i+j}$ for some $j>1$ and all $i<0$, then
  $A\in\frak g^j$. Since this fact will not be needed in what follows,
  we only sketch briefly how this is proved:

  Suppose that for some $0<\ell<j$ we have $A\in\frak g^\ell$ with
  non--zero image in $\gr_\ell(\frak g)$, such that $\ad(A)(\frak
  g^i)\subset\frak g^{i+j}$ for all $i<0$. Then $\ad(A)$ induces a map
  $\ph$ on $\gr(\frak g)$ which is homogeneous of degree $j$ and
  satisfies $\ph([X,Y])=[\ph(X),Y]+[X,\ph(Y)]$ for all
  $X,Y\in\gr(\frak g)$ by the Jacobi identity. Since $\gr(\frak g)$ is
  the full prolongation of $(\frak m,\gr_0(\frak g))$ we conclude that
  $\ph$ must coincide with the adjoint action of some element of
  $\gr_j(\frak g)$. Taking a representative $B\in\frak g^j$ for this
  element, we conclude that $A-B\in\frak g^\ell$ has the property that
  its adjoint action maps $\frak g^i$ to $\frak g^{i+j+1}$ for all
  $i<0$ and also has non--zero image in $\gr_\ell(\frak g)$. Assuming
  that already $A$ has this property, we can iterated this until we
  reach an element $A\in \frak g^\ell$ with non--zero image in
  $\gr_\ell(\frak g)$ such that $\ad(A)(\frak g^i)\subset\frak
  g^{i+\nu+1}$, where $\frak g^\nu$ is the smallest non--trivial
  filtration component of $\frak g$.

  At this stage, the map $\ph$ on $\gr(\frak g)$ induced by $\ad(A)$
  is homogeneous of degree $\nu+1$. But for degrees $s>\nu$, we have
  $\gr_s(\frak g)=\{0\}$, so the condition on being the full
  prolongation actually says that $\ph=0$. Thus we see that
  $\ad(A)(\frak g^i)\subset\frak g^{i+\nu+2}$ and iterating once more,
  we reach $\ad(A)(\frak g^i)\subset\frak g^{i+\mu+\nu+1}=\{0\}$ for
  all $i<0$. But this then says that $A$ lies in the center of $\frak
  g$ and thus spans a one--dimensional ideal in $\frak g$. Since this
  ideal is contained in $\frak g^\ell\subset\frak g^0$, infinitesimal
  effectivity leads to a contradiction.
\end{remark}

The conditions from Definition \ref{def2.5} can be neatly phrased in
terms of Lie algebra cohomology. Since $\frak m\subset\gr(\frak g)$ is
a Lie subalgebra, the restriction of the adjoint action makes
$\gr(\frak g)$ into a graded module over the graded Lie algebra $\frak
m$. Now the standard complex for computing the Lie algebra cohomology
$H^*(\frak m,\gr(\frak g))$ has the chain groups $C^k(\frak
m,\gr(\frak g)):=\La^k\frak m^*\otimes\gr(\frak g)$ of $k$--linear,
alternating maps $\frak m^k\to\gr(\frak g)$. The usual homogeneity of
maps defines a grading $C^k(\frak m,\gr(\frak g))=\oplus_\ell
C^k(\frak m,\gr(\frak g))_\ell$. Here we say that $\ph:\frak
m^k\to\gr(\frak g)$ is homogeneous of degree $\ell$ if and only if for
$X_j\in\frak m_{i_j}$ with $j=1,\dots,k$ and $i_j\in\{-\mu,\dots,-1\}$
for all $j$, we have
$\ph(X_1,\dots,X_k)\in\gr_{i_1+\dots+i_k+\ell}(\frak g)$.

The standard differential $\partial:C^k(\frak m,\gr(\frak g))\to
C^{k+1}(\frak m,\gr(\frak g))$ is defined by the usual formula 
\begin{equation}
  \label{partial-def}
  \begin{aligned}
    \partial\ph(X_0,\dots,X_k):=&
    \tsum_{i=0}^k(-1)^i[X_i,\ph(X_0,\dots,\widehat{X_i},\dots,
    X_k)]+\\
    &\tsum_{i<j}(-1)^{i+j}\ph([X_i,X_j],X_0,\dots,\widehat{X_i},\dots,
\widehat{X_j},\dots X_k)
  \end{aligned}
\end{equation}
for $X_0,\dots,X_k\in\frak m$ with hats denoting omission. Since the
brackets preserve the grading, it readily follows that if $\ph$ is
homogeneous of degree $\ell$, then the same holds for
$\partial\ph$. This implies that the cohomology spaces inherit a
grading which we denote by $H^k(\frak m,\gr(\frak g))=\oplus_\ell
H^k(\frak m,\gr(\frak g))_\ell$. Using this, the following result is
well known, we include the simple proof for completeness.

\begin{prop}\label{prop2.5}
  Let $(\frak g,P)$ be an admissible pair. Then $\gr(\frak g)$ is the
  full prolongation of $(\frak m,\gr_0(\frak g))$ (respectively of
  $\frak m$) if and only if $H^1(\frak m,\gr(\frak g))_\ell=0$ for all
  $\ell>0$ (respectively for all $\ell\geq 0$).
\end{prop}
\begin{proof}
  Let $\ph:\frak m\to\gr(\frak g)$ be homogeneous of degree $\ell\geq
  0$, so $\ph\in C^1(\frak m,\gr(\frak g))_\ell$. Then $\partial\ph=0$
  exactly says that $0=[X,\ph(Y)]-[Y,\ph(X)]-\ph([X,Y])$ for all
  $X,Y\in\frak m$. If $\ell=0$, then $\ph$ has values in $\frak
  m\subset\gr(\frak g)$ and this equation exactly says that
  $\ph\in\frak{der}_{\gr}(\frak m)$. For $\ell>0$ it exactly boils
  down to the condition used in Definition \ref{def2.5}. On the other
  hand $C^0(\frak m,\gr(\frak g))_\ell=\gr_\ell(\frak g)$ and
  $\ph=\partial Z$ exactly says that $\ph(X)=[X,Z]$, so
  $\ph=-\ad(Z)|_{\frak m}$ and the claim follows.
\end{proof}

\subsection{Examples}\label{2.6}
\textbf{1.~Vanishing prolongation}: Let $\frak m=\frak
m_{-\mu}\oplus\dots\oplus\frak m_{-1}$ be any nilpotent graded Lie
algebra and fix a Lie subalgebra $\frak
g_0\subset\frak{der}_{\gr}(\frak m)$. Then $\frak m\oplus\frak g_0$
naturally is a Lie algebra via
$[(X,A),(Y,B)]:=([X,Y]+A(Y)-B(X),[A,B])$. Assume that $\frak
m\oplus\frak g_0$ is the full prolongation of $(\frak m,\frak g_0)$,
i.e.~that there is no non-zero linear map $\ph:\frak m\to\frak
m\oplus\frak g_0$ which is homogeneous of some positive degree such
that $\ph([X,Y])=[\ph(X),Y]+[X,\ph(Y)]$. Let $G_0$ be a Lie group with
Lie algebra $\frak g_0$ such that the inclusion $\frak
g_0\hookrightarrow\frak{der}_{\gr}(\frak m)$ integrates to a
homomorphism $G_0\to\Aut_{\gr}(\frak m)$. (For example, one may take
the connected virtual Lie subgroup in $\Aut_{\gr}(\frak m)$
corresponding to $\frak g_0$.) Then the filtration defined by $\frak
g^i:=(\oplus_{j\geq i}\frak m_j)\oplus\frak g_0$ for $i=-\mu,\dots,0$,
evidently makes $(\frak g:=\frak m\oplus\frak g_0,G_0)$ into an
admissible pair. The associated graded Lie algebra $\gr(\frak g)$ then
of course is just $\frak m\oplus\frak g_0$ and thus coincides with the
full prolongation of $(\frak m,\gr_0(\frak g))$.

While this is a very simple situation, it covers several interesting
cases. On the one hand, consider a fundamental graded Lie algebra
$\frak m=\oplus_{i=-\mu}^{-1}\frak m_i$ and a positive definite inner
product $b$ on $\frak m_{-1}$. Then as observed in Remark
\ref{rem2.1}(2), any graded derivation of $\frak m$ is determined by
its restriction to $\frak m_{-1}$. Thus we may form $\frak
g_0:=\frak{der}_{\gr}(\frak m)\cap\frak{so}(\frak m_{-1})$ and a
theorem of Morimoto (see \cite{Morimoto:subRiem}) shows that $(\frak
m,\frak g_0)$ has vanishing prolongation. Hence the models for
sub--Riemannian structures of constant type all fall into this
category.

On the other hand, consider a real vector space $\frak m_{-1}$ of even
dimension endowed with a non--degenerate, skew--symmetric bilinear
form $b$. View this as a linear surjection $\La^2\frak m_{-1}\to\Bbb
R$ and define $\frak m_{-2}:=\La^2_0\frak m_{-1}$ to be its kernel. On
the other hand, we can view $b$ as defining a linear isomorphism
$\frak m_{-1}\to\frak m_{-1}^*$ and the inverse of this isomorphism
defines an element $\tilde b\in\La^2\frak m_{-1}$ such that $b(\tilde
b)=1$. Now we define a bracket $[\ ,\ ]:\frak m_{-1}\x\frak
m_{-1}\to\frak m_{-2}$ by $[X,Y]:=X\wedge Y-b(X,Y)\tilde b$. This is
evidently skew--symmetric and since the Jacobi identity is trivially
satisfied, it makes $\frak m:=\frak m_{-2}\oplus\frak m_{-1}$ into a
fundamental graded Lie algebra. Now we define $\frak{csp}(\frak
m_{-1})$ to be the Lie algebra of all endomorphisms $A$ of $\frak
m_{-1}$ for which there is a number $\la\in\Bbb R$ such that for all
$X,Y\in\frak m_{-1}$ we get $b(AX,Y)+b(X,AY)=\la b(X,Y)$. It is easy
to see that for $A\in\frak{csp}(\frak m_{-1})$, the induced map on
$\La^2\frak m_{-1}$ preserves the direct sum decomposition $\frak
m_{-2}\oplus\Bbb R\cdot\tilde b$. Hence there is an induced
endomorphism of $\frak m_{-2}$ and one immediately verifies that
together with $A$, this defines a graded derivation of $\frak m$.

One shows that this construction actually defines an isomorphism
between $\frak{csp}(\frak m_{-1})$ and $\frak{der}_{\gr}(\frak m)$,
and one may take this full algebra to be $\frak g_0$. Now of course
$\frak{csp}(\frak m_{-1})$ is reductive with one--dimensional center
and semisimple part $\frak g_0^{ss}:=\frak{sp}(\frak m_{-1})$. Using
this, the beginning of the standard complex computing the Lie algebra
cohomology $H^*(\frak m,\frak m\oplus\frak g_0)$ can be analyzed using
representation theory of $\frak g_0^{ss}$. This is carried out in the
thesis \cite{deZanet}, and in particular it is shown in Proposition 11
of that reference that $(\frak m,\frak g_0)$ has vanishing
prolongation. Hence in this case, we can simply put
$G_0=\Aut_{\gr}(\frak m)\cong CSp(\frak m_{-1})$ to obtain an
appropriate admissible pair $(\frak g,G_0)$. It turns out that this is
the model for the unique type of generic distributions of even rank
$n=2m$ in manifolds of dimension $\frac{n(n+1)}{2}-1$, see
\cite{deZanet}.

\smallskip

\textbf{2.~Parabolics}: Let $G$ be a Lie group, whose Lie algebra
$\frak g$ is semisimple, and let $\frak p:=\frak g^0$ be a parabolic
subalgebra. One characterization of parabolic subalgebras is that the
annihilator of $\frak p$ with respect to the Killing form $B$ of
$\frak g$ is contained in $\frak p$ and coincides with the nilradical
of $\frak p$. Denoting this by $\frak g^1\subset\frak g^0$, one
defines $\frak g^2:=[\frak g^1,\frak g^1]$ and inductively $\frak
g^{i+1}=[\frak g^i,\frak g^1]$. This defines a filtration $\frak
g^0\supset\frak g^1\supset\dots\supset\frak g^\nu\supset\frak
g^{\nu+1}=\{0\}$, where we agree that $\frak g^\nu$ is the last non--zero
term. Then for $i<0$, one defines $\frak g^i$ as the annihilator of
$\frak g^{-i+1}$ under the Killing form. The resulting filtration then
has the form 
$$ \frak g=\frak g^{-\nu}\supset\frak
g^{-\nu+1}\supset\dots\supset\frak g^0\supset\frak
g^1\supset\dots\supset\frak g^\nu
$$ and this makes $\frak g$ into a filtered Lie algebra.  It is well
known that parabolic subalgebras can be equivalently described in
terms of gradings on the Lie algebra $\frak g$ and essentially finding
such a grading amounts to choosing a Cartan subalgebra contained in
$\frak g^0$, see Section 3.2 in \cite{book}. More precisely, there are
subspaces $\frak g_i\subset\frak g$ for $i=-\nu,\dots,\nu$ such that
$[\frak g_i,\frak g_j]\subset\frak g_{i+j}$ and such that for each
$j=-\nu,\dots,\nu$, we have $\frak g^j=\oplus_{i\geq j}\frak g_i$. In
particular this shows that $\gr(\frak g)$ is isomorphic to $\frak g$
as a Lie algebra, but conceptually it is better to distinguish between
the two. Assuming condition (A) in Definition \ref{def2.1}, i.e.~that
none of the simple ideals of $\frak g$ is contained in $\frak g^0$,
then it is well known that condition (B) from that Definition is
automatically satisfied, too.

Now choose a subgroup $P\subset G$ that lies between the normalizer of
$\frak g^0$ in $G$ and its connected component of the identity. Then
we see that $(\frak g,P)$ is an admissible pair. For this specific
case, the Lie algebra cohomology $H^*(\frak m,\gr(\frak g))$ can be
computed using Kostant's theorem (see \cite{Kostant}) for complex
$\frak g$. Via complexification, this also handles the real case, and
it turns out that for almost all cases $\frak g$ is the full
prolongation of $(\frak m,\gr_0(\frak g))$. Basically, this result go
back to N.~Tanaka in \cite{Tanaka:simple}, see also K.~Yamaguchi's
article \cite{Yamaguchi} and Section 3.3.7 of \cite{book}. It is also
possible to characterize the cases in which $\gr(\frak g)$ is the full
prolongation of $\frak m$, see \cite{Yamaguchi} and Proposition 4.3.1
of \cite{book}.

\smallskip

\textbf{3.~Algebras related to (systems of) ODEs}:

Consider the one--dimensional projective space $\Bbb RP^1$, realized
as the qoutient $(\Bbb R^2\setminus\{0\})/\sim$, where $x\sim y$ iff
there is a number $t\in\Bbb R$ such that $y=tx$. Via the standard
action of $SL(2,\Bbb R)$ on $\Bbb R^2$, this is identified with the
homogeneous space $SL(2,\Bbb R)/B$, where $B$ is the stabilizer of a
distinguished line in $\Bbb R^2$. For $k\in\Bbb Z$ and $m\geq 1$, we
can define $\Cal O(k)^m$ as $((\Bbb R^2\setminus\{0\})\times\Bbb
R^m)/\sim_k$, where $(x,v)\sim_k (y,w)$ iff there is a number
$t\in\Bbb R$ such that $y=tx$ and $w=t^kv$. The projection onto the
first factor gives rise to smooth map $\Cal O(k)^m\to\Bbb RP^1$ which
makes $\Cal O(k)^m$ into a vector bundle of rank $m$ over $\Bbb
RP^1$. For $k=-1$ and $m=1$, this produces the tautological line
bundle $\Cal O(-1)$ over $\Bbb RP^1$. By construction, there is a
natural action of the group $SL(2,\Bbb R)\x GL(m,\Bbb R)$ on $\Cal
O(k)^m$ which extends the action on $\Bbb RP^1$ via the first factor.

From the definition it is also clear that smooth sections of the
bundle $\Cal O(k)^m$ can be identified with smooth maps $\ph:\Bbb
R^2\setminus\{0\}\to\Bbb R^m$ which are homogeneous of degree $k$ in
the sense that $\ph(tx)=t^k\ph(x)$. For $k>0$, we can in particular
consider the space $V_k^m:=S^k\Bbb R^{2*}\otimes\Bbb R^m$ of $\Bbb
R^m$--valued homogeneous polynomials of degree $k$ on $\Bbb R^2$. Any
such polynomial defines a global section of the bundle $\Cal O(k)^m$.
Using this, we can define a map from $\Bbb RP^1\x V_k^m\to J^k(\Cal
O(k)^m)$ of $k$--jets of local smooth sections of the bundle $\Cal
O(k)^m$, by sending $(\ell,\ps)$ to the $k$--jet of the global section
of $\Cal O(k)^m$ determined by $\ps\in V_k^m$ at the point
$\ell\in\Bbb RP^1$. It is elementary to verify that this construction
defines an isomorphism $\Bbb RP^1\x V_k^m\to J^k(\Cal O(k)^m)$ of
vector bundles and thus a natural trivialization of this specific jet
bundle. Also, this trivialization is compatible with the natural
actions of the group $SL(2,\Bbb R)\x GL(m,\Bbb R)$ on both
sides. Finally, via the trivialization, for a section $\si\in\Ga(\Cal
O(k)^m)$, the $k$--jet $j^k\si$ defines a smooth function $\Bbb
RP^1\to V_k^m$. Requiring this function to have vanishing derivative
can be viewed as a differential equation of order $k+1$ on sections of
$\Cal O(k)^m$. It is easy to see that in standard local adapted jet
coordinates, this is expressed by the trivial system $y_i^{(k+1)}=0$
for $i=1,\dots,m$.

Now we define $G:=(SL(2,\Bbb R)\x GL(m,\Bbb R))\ltimes V_k^m$, the
semi--direct product of the group $SL(2,\Bbb R)\x GL(m,\Bbb R)$ with
its representation $V_k^m$. This naturally acts on $\Bbb RP^1\x V_k^m$
with $SL(2,\Bbb R)\x GL(m,\Bbb R)$ acting as described above and
elements of $V_k^m$ acting by translation in the second factor. By
construction, $G$ acts transitively on $\Bbb RP^1\x V_k^m$ and its
action preserves the system of ODEs constructed above. Denoting by
$\ell_0\in\Bbb RP^1$ the line stabilized by $B$, the isotropy group of
$(\ell_0,0)$ under the $G$--action is visibly given by $(B\x GL(m,\Bbb
R))\ltimes\{0\}\subset G$.

On the level of Lie algebras, we get $\frak g=(\frak{sl}(2,\Bbb
R)\x\frak {gl}(m,\Bbb R))\oplus V_k^m$ (semi--direct sum) and $\frak
g^0=\frak p=\frak b\x\frak{gl}(m,\Bbb R)$. Now $\frak{sl}(2,\Bbb R)$
carries the canonical $B$--invariant filtration defined by
$\frak{sl}(2,\Bbb R)\supset\frak b\supset[\frak b,\frak b]$, and we
define a filtration on $\frak{sl}(2,\Bbb R)\x\frak {gl}(m,\Bbb R)$ by
simply taking products with $\frak{gl}(m,\Bbb R)$. On the other hand,
the representation $S^k\Bbb R^{2*}=V_k^1$ has an obvious
$B$--invariant filtration induced from the standard weight
decomposition. We fix the degrees in such a way that the component of
degree $0$ is trivial, while for $i>0$, the component of degree $-i$
to be spanned by the weight spaces corresponding to the $i$ largest
weights of the representation $V_k^1$. Since the action of $B$ never
lowers weights, this filtration is $B$--invariant. Taking the tensor
product with $\Bbb R^m$, we arrive at a filtration of $V_k^m$, which
is invariant under $B\x GL(m,\Bbb R)$. Here the dimensions of the
filtration components grow by $m$ in each step.

Taking these together, we obtain a filtration of $\frak g$, with
$\mu=k+1$ and $\nu=1$, i.e.~of the form $\frak g=\frak
g^{-k-1}\supset\dots\supset\frak g^0=\frak p\supset\frak
g^1\supset\{0\}$. For the associated graded we get $\frak
m=\oplus_{i=-k-1}^{-1}\frak m_i$. The dimension of $\frak m_{-1}$ is
$m+1$ (with one dimension corresponding to the negative root space in
$\frak{sl}(2,\Bbb R)$ and the rest corresponding to the tensor product
of the highest weight space in $S^k\Bbb R^{2*}$ with $\Bbb R^m$),
while all lower components of $\frak m$ have dimension $m$. The
subalgebra $\gr_0(\frak g)$ is isomorphic to $(\frak b/[\frak b,\frak
b])\oplus\frak{gl}(m,\Bbb R)$, while $\gr_1(\frak g)$ is
one--dimensional and spanned by the positive root space in
$\frak{sl}(2,\Bbb R)$. From this one immediately verifies that $(\frak
g,P)$ is an admissible pair in the sense of Definition \ref{def2.2}.

\smallskip

It turns out that $\gr(\frak g)$ is \textit{not} the full prolongation
of $(\frak m,\gr_0(\frak g))$ for all possible choices of $k$ and
$m$. Indeed, if $k=1$, then for each $m\geq 1$ the pair $(\frak
m,\gr_0(\frak g))$ is isomorphic to the non--positive part in the
grading of $\frak{sl}(m+2,\Bbb R)$ corresponding to the first two
simple roots. As discussed in Example 2 above, this implies that that
full prolongation of $(\frak m,\gr_0(\frak g))$ is $\frak{sl}(m+2,\Bbb
R)$, whose dimension is strictly larger than $\dim(\frak g)$. This
corresponds to the fact that second order ODEs and systems of second
order ODEs are equivalent to parabolic geometries via the concept of
path geometries, compare with Sections 4.4.3 and 4.4.4 in \cite{book}
and Section 4.7 in \cite{twistor}.

Similarly, if $k=2$ and $m=1$, then $(\frak m,\gr_0(\frak g))$ is
isomorphic to the non--positive part of the grading of
$\frak{sp}(4,\Bbb R)\cong\frak{so}(3,2)$ determined by both simple
roots. Again by Example 2, we conclude that the full prolongation of
$(\frak m,\gr_0(\frak g))$ is $\frak{sp}(4,\Bbb R)$ in this case, and 
this is strictly larger than $\frak g$. This corresponds to Chern's
classical result \cite{Chern} on the geometry of a single third order
ODE up to contact transformations, which in modern language says that
this can be equivalently described as a parabolic geometry.

For all other choices of $k$ and $m$ (i.e.~if either $k\geq 3$ or
$k=2$ and $m\geq 2$), it turns out that $\gr(\frak g)$ is the full
prolongation of $(\frak m,\gr_0(\frak g))$. This is shown in
\cite{DKM}, a short, direct proof based on the cohomological
interpretation from Proposition \ref{prop2.5} can be found in
\cite{CDT}.

\begin{remark}\label{rem2.6}
  Suppose that we have given $\frak m$ and a Lie group $G_0$ together
  with an infinitesimally injective homomorphism
  $G_0\to\Aut_{\gr}(\frak m)$, so there is the concept of filtered
  $G_0$--structures on filtered manifolds which are regular of type
  $\frak m$. The basic philosophy of this article is that an
  infinitesimal homogeneous model for such structures is known in
  advance. If this is not the case, and one has to start from $\frak
  m$ and $G_0$ only, there is a construction principle for such a
  candidate as follows. By assumption, we can view the Lie algebra
  $\frak g_0$ of $G_0$ as a Lie subalgebra of $\frak{der}_{\gr}(\frak
  m)$. Using this, we can make $\frak m\oplus\frak g_0$ into a graded
  Lie algebra. Explicitly, we define the bracket on $\frak
  m\oplus\frak g_0$ by
$$ 
[(X,A),(Y,B)]:=([X,Y]_{\frak m}+A(Y)-B(X),[A,B]_{\frak g_0}).
$$ 
Moreover, the given action on $\frak m$ and the adjoint action on
$\frak g_0$ define an action of $G_0$ on $\frak m\oplus\frak g_0$ by
Lie algebra automorphisms. Following Tanaka, one can now inductively
add components $\frak g_i$ for $i\geq 0$ which make $\frak
m\oplus\frak g_0\oplus \oplus_{i>0}\frak g_i$ into a graded Lie
algebra $\pr(\frak m,\frak g_0)$ which is maximal in a certain sense,
see \cite{Zelenko} for details. This is called the \textit{Tanaka
  prolongation} of $(\frak m,\frak g_0)$. Basically, for each $i$, one
defines $\frak g_i$ as those elements in the space of linear maps from
$\frak m$ to $\frak m\oplus\oplus_{0\leq j<i}\frak g_j$ that are
homogeneous of degree $i$ and satisfy a derivation property.

Now $(\frak m,\frak g_0)$ is said to be of \textit{finite type}, if
this process stops after finitely many steps and thus $\pr(\frak
m,\frak g_0)$ is a finite dimensional graded Lie algebra. Let us
denote by $\nu>0$ the maximal index for which $\frak
g_\nu\neq\{0\}$. We can then put $\frak g:=\pr(\frak m,\frak g_0)$ and
endow it with the filtration induced by the grading, so that $\frak
g\cong\gr(\frak g)$. By construction, $\gr(\frak g)$ then is the full
prolongation of $(\frak m,\frak g_0)$. It also follows readily that
conditions (A) and (B) from Definition \ref{def2.2} (which do not
depend on the group $P$) are automatically satisfied.

To obtain an admissible pair, it thus remains to find a Lie group $P$
with Lie algebra $\oplus_{i\geq 0}\frak g_i$ and an action of $P$ on
$\frak g$ that satisfies property (iii) from Definition \ref{def2.2}.
The basic idea here is to first use the construction of the
prolongation to lift the obvious action of $\frak g_0$ on $\frak g$ to
a group action of $G_0$. Since this action preserves the grading, we
can restrict it to $\frak p_+:=\oplus_{i=1}^\nu\frak g_i$, which
clearly is a nilpotent Lie subalgebra of $\frak g$. Now define $P_+$
to be the simply connected group with Lie algebra $\frak p_+$ and try
to lift the $G_0$--action to an action on $P_+$ by group
automorphisms. If this works, one defines $P$ as the semi--direct
product of $G_0$ and $P_+$, and then one can try to construct a
$P$--action on $\frak g$ from the action of $G_0$ and the given action
of $\frak p_+$. Since this is not the approach we have chosen, we do
not study the precise conditions under which this is possible.
\end{remark}

\section{On normalization conditions}\label{3}

Looking at the construction of the filtered $G_0$--structure
underlying a Cartan geometry in the proof of Theorem \ref{thm2.4}, it
is evident that this underlying structure can never determine the
Cartan geometry uniquely. In fact, one can add any form of positive
homogeneity to a given Cartan connection without changing the induced
underlying structure, see Proposition \ref{prop4.3} for details. To
remove this freedom, one has to impose a normalization condition on
the curvature of the Cartan connection. As we shall see, this is a
purely algebraic problem, which we discuss in this section.

\subsection{The concept of a normalization condition}\label{3.1}
Let us start with an admissible pair $(\frak g,P)$ in the sense of
Definition \ref{def2.2} and let $\{\frak g^i\}_{i=-\mu}^\nu$ be the
corresponding filtration of $\frak g$, so $\frak p=\frak g^0$. Now for
each $k\geq 0$, we can consider the space $L(\La^k(\frak g/\frak
p),\frak g)$ of alternating $k$--linear maps $(\frak g/\frak
p)^k\to\frak g$. Observe that these can be viewed equivalently as
alternating $k$--linear maps $(\frak g)^k\to\frak g$ which vanish
whenever one of their entries lies in the subspace $\frak
p\subset\frak g$.

Now for such maps, there is an obvious notion of homogeneity (in the
sense of filtrations). We say that $\alpha$ is homogeneous of degree
$\geq\ell$ if and only if for any $X_j\in\frak g^{i_j}$ with
$j=1,\dots,k$ and $i_j<0$ for all $j$, we have 
$$
\al(X_1+\frak p,\dots,X_k+\frak p)\in\frak g^{i_1+\dots+i_k+\ell}.
$$
Of course, the maps with this property form a linear subspace
$L(\La^k(\frak g/\frak p),\frak g)^\ell\subset L(\La^k(\frak g/\frak
p),\frak g)$ and by construction, these spaces form a filtration of
the (finite dimensional) vector space $L(\La^k(\frak g/\frak p),\frak
g)$. Now we can nicely describe the associated graded to this filtered
vector space.

\begin{lemma}\label{lem3.1} 
  For $(\frak g,P)$ as above consider the associated graded $\gr(\frak
  g)$ and as before define $\frak m:=\oplus_{i=-\mu}^{-1}\gr_i(\frak
  g)$. Then the quotient $L(\La^k(\frak g/\frak p),\frak
  g)^\ell/L(\La^k(\frak g/\frak p),\frak g)^{\ell+1}$ can be naturally
  identified with the space $C^k(\frak m,\gr(\frak
  g))_\ell=L(\La^k\frak m,\gr(\frak g))_\ell$ of $k$--cochains which
  are homogeneous of degree $\ell$.
\end{lemma}
\begin{proof}
  This is a direct verification, compare with Section 3.1.1 of
  \cite{book}. Given a map $\al\in L(\La^r(\frak g/\frak p),\frak
  g)^\ell$ and elements $\tau_j\in\frak m_{i_j}$, put
  $s:=i_1+\dots+i_k+\ell$. Choosing a representative $X_j\in\frak
  g^{i_j}$ of $\tau_j$ for each $j$, we have
  $\al(X_1,\dots,X_k)\in\frak g^s$ by definition, so we can consider
  its class in $\gr_s(\frak g)$. Any other representative $\tilde X_j$
  for $\tau_j$ is of the form $X_j+Y_j$ with $Y_j\in\frak
  g^{i_j+1}$. Homogeneity of $\al$ implies that
  $\al(X_1,\dots,Y_j,\dots,X_k)\in\frak g^{s+1}$ so the class of
  $\al(X_1,\dots,X_k)$ in $\gr_s(\frak g)$ is independent of the
  choice of representatives. Otherwise put, we have associated to
  $\al$ a well defined linear map $\frak m_{i_1}\x\dots\x\frak
  m_{i_k}\to\gr_s{\frak g}$. Taking these maps for all possible
  choices of the $i_j$ together, we obtain a well defined map $\frak
  m^k\to\gr(\frak g)$ induced by $\al$, which by construction is
  alternating and homogeneous of degree $\ell$.

  The construction readily implies that this construction actually
  defines a linear map $L(\La^k(\frak g/\frak p),\frak g)^\ell\to
  L(\La^k\frak m,\gr(\frak g))_\ell$. Moreover, $\al$ lies in the
  kernel of this map if and only if for $X_j\in\frak g^{i_j}$ as
  above, one always has $\al(X_1,\dots,X_k)\in\frak
  g^{i_1+\dots+i_k+\ell+1}$ and thus if and only if $\al$ is
  homogeneous of degree $\geq\ell+1$. So it remains to prove that our
  map is surjective. To see this, we put $W_\nu=\frak g^\nu$ and for
  each $i=-\mu,\dots,\nu-1$, we choose a linear subspace
  $W_i\subset\frak g^i$, which is complementary to $\frak
  g^{i+1}$. Then clearly $\frak g=\oplus_{i=-\mu}^\nu W_i$ as a vector
  space. On the other hand, the canonical projection restricts to a
  linear isomorphism $W_i:\to\gr_i(\frak g)$. Making these choices, we
  have thus constructed a linear isomorphism $\ph^W:\frak
  g\to\gr(\frak g)$, which restricts to a linear isomorphism between
  $\frak g^i$ and $\oplus_{j\geq i}\gr_j(\frak g)$ for each $i$. (The
  inverse of such an isomorphism is commonly called a
  \textit{splitting of the filtration}.) In particular, we get an
  induced isomorphism $\underline{\ph}^W:\frak g/\frak p\to\frak m$
  which also is compatible with the grading.

  Now suppose we have given a $k$--linear alternating map $\be:\frak
  m^k\to\gr(\frak g)$, which is homogeneous of degree $\ell$. Then
  defining $\al:(\frak g/\frak p)^k\to\frak g$ as
  $(\ph^W)^{-1}\o\be\o(\underline{\ph}^W)^k$, it is easy to verify
  that $\al$ is homogeneous of degree $\geq\ell$ and maps to $\be$.
\end{proof}

\begin{definition}\label{def3.1}
  We denote by $\gr_{\ell}:L(\La^k(\frak g/\frak p),\frak g)^\ell\to
  C^k(\frak m,\gr(\frak g))_\ell$ the linear map described in Lemma
  \ref{lem3.1}.
\end{definition}

Initially, we will mainly need this in the case that $k=2$ and
$\ell>0$. So in this case, we associate to a map $\al:\La^2(\frak
g/\frak p)\to\frak g$ which is homogeneous of degree $\geq\ell$ in the
filtration sense the map $\gr_{\ell}(\al):\La^2\frak m\to\gr(\frak
g)$, which is homogeneous of degree $\ell$. Observe that for
$X\in\frak g^i$ and $Y\in\frak g^j$ with $i,j<0$ we have
\begin{equation}\label{grell-def}
  \gr_{\ell}(\al)(\gr_i(X),\gr_j(Y))=\al(X+\frak
  p,Y+\frak p)+\frak g^{i+j+\ell+1}\in\gr_{i+j+\ell}(\frak g).
\end{equation}

Observe also that a filtration on a vector space $V$ induces a
filtration on any linear subspace $W\subset V$, by simply defining
$W^i:=W\cap V^i$. Consequently, we can form the associated graded
vector space to $W$ with respect to this filtration and for each $i$
naturally view $\gr_i(W)$ as a linear subspace of $\gr_i(V)$. Armed
with this observation, we can now formulate the following crucial
definition.

\begin{definition}\label{def-norm}
  Let $(\frak g,P)$ be an admissible pair as in Definition
  \ref{def2.2}, let $\gr(\frak g)$ be the associated graded to $\frak
  g$ and put $\frak m:=\oplus_{i=-\mu}^{-1}\gr_i(\frak g)$. Then a
  \textit{normalization condition} for $(\frak g,P)$ is a
  $P$--invariant linear subspace $\Cal N\subset L(\La^2(\frak g/\frak
  p),\frak g)$ such that for each $\ell>0$ the subspace
  $\gr_{\ell}(\Cal N)\subset C^2(\frak m,\gr(\frak g))_\ell$ is
  complementary to the image of the linear map $\partial:C^1(\frak
  m,\gr(\frak g))_\ell\to C^2(\frak m,\gr(\frak g))_\ell$ defined in
  equation \eqref{partial-def}.
\end{definition}

\begin{remark}\label{rem3.1}
  (1) Observe that the defining properties of a normalization
  conditions take place on two different levels. The condition on
  $P$--invariance concerns the subspace $\Cal N$ in the filtered
  vector space $L(\La^2(\frak g/\frak p),\frak g)$, on which there is
  no well defined Lie algebra cohomology differential. The
  complementarity condition, on the other hand, refers to the image of
  the filtration components of $\Cal N$ in the associated graded, on
  which a substantial part of the $P$--action is lost.

  (2) There is no reason to expect that normalization conditions exist
  for all admissible pairs $(\frak g,P)$. Even though only natural
  ingredients are used in Definition \ref{def-norm}, one has to keep
  in mind that invariant subspaces do not admit invariant complements
  in general, in particular, if $P$ contains a large solvable
  part. However, while existence of normalization conditions is known
  in many cases of interest (see in particular the examples in Section
  \ref{3.4} below), I am not aware of proofs of non--existence of a
  normalization condition in the literature. Let us also remark here,
  that the construction of a canonical absolute parallelism in
  \cite{Zelenko} works without assumptions on invariance of
  normalization conditions, so this can always be applied.
\end{remark}

\subsection{Negligible submodules}\label{3.2}
In what follows, a normalization conditions will describe the allowed
values for the curvature function of a normal Cartan connection. From
examples like parabolic geometries it is known that for some
structures one may pass from the full Cartan curvature to a simpler
geometric object, which still defines a complete obstruction against
local flatness of the geometry. We next introduce the algebraic
background for results of this type.

\begin{definition}\label{def-neg}
Let $(\frak g,P)$ be an admissible pair and let $\Cal N\subset
L(\La^2(\frak g/\frak p),\frak g)$ be a normalization condition for
$(\frak g,P)$. Then a \textit{negligible submodule} in $\Cal N$ is a
$P$--invariant subspace $\tilde{\Cal N}\subset\Cal N$ such that for
each $\ell>0$ the image $\gr_\ell(\tilde{\Cal N})\subset C^2(\frak
m,\gr(\frak g))_\ell$ has trivial intersection with
$\ker(\partial)$. 

We call $\tilde{\Cal N}$ a \textit{maximal} negligible submodule iff
$\gr_\ell(\tilde{\Cal N})$ is complementary to $\ker(\partial)$ for
all $\ell$.
\end{definition}

As discussed in \ref{3.1}, the filtration on $L(\La^2(\frak g/\frak
p),\frak g)$ can be restricted to any linear subspace. In particular,
for a negligible submodule $\tcn\subset\Cal N$, we get
$\tcn^\ell\subset\Cal N^\ell$ for each $\ell>0$. Of course, these
filtrations are preserved by the $P$--action on both modules. In
particular, we get an induced filtration on the quotient modules $\Cal
N/\tcn$, which again is $P$--invariant. In particular, we can form
$\gr_{\ell}(\Cal N/\tcn)$ for each $\ell$.

\begin{prop} 
Let $(\frak g,P)$ be an admissible pair and let $\Cal N$ be a
normalization condition for $(\frak g,P)$. 

(1) For each $\ell>0$, the subspace $\ker(\partial)\cap\gr_{\ell}(\Cal
  N)$ of $C^2(\frak m,\gr(\frak g))$ is linearly isomorphic to the
  degree--$\ell$ component $H^2(\frak m,\gr(\frak g))_\ell$ in the
  second cohomology space. 

  (2) If $\tcn\subset\Cal N$ is a maximal negligible submodule then
  for each $\ell>0$, projection to the associated graded together with
  the map from (1) induces a linear isomorphism $\gr_\ell(\Cal
  N/\tcn)\to H^2(\frak m,\gr(\frak g))_\ell$.
\end{prop}
\begin{proof}
  (1) For $\ph\in C^2(\frak m,\gr(\frak g))_\ell$ with
  $\partial\ph=0$, let us denote by $[\ph]$ the cohomology class of
  $\ph$ in $H^2(\frak m,\gr(\frak g))$. By definition of a
  normalization condition, we can write $\ph$ as $\ph_1+\ph_2$ with
  $\ph_1\in\gr_\ell(\Cal N)$ and $\ph_2\in\im(\partial)$. Since
  $\partial\ph_2=0$, we get $\partial\ph_1=0$, and of course
  $[\ph_1]=[\ph]$, so surjectivity is proved. But if
  $\ph\in\gr_\ell(\Cal N)$ satisfies $\partial\ph=0$ and $[\ph]=0$,
  then $\ph\in\im(\partial)$ and hence $\ph=0$ by definition of a
  normalization condition.

  (2) For $\al\in\Cal N^\ell$ consider $\gr_\ell(\al)\in C^2(\frak
  m,\gr(\frak g))_\ell$. By definition of a maximal negligible
  submodule, we can write $\gr_\ell(\al)$ as the sum of an element of
  $\gr_\ell(\tcn)$ and an element of $\ker(\partial)$. Otherwise put,
  there is an element $\be$ in $\tcn^\ell$ such that
  $\gr_{\ell}(\al-\be)\in\ker(\partial)$, so we can form
  $[\gr_{\ell}(\al-\be)]\in H^2(\frak m,\gr(\frak g))_\ell$. If
  $\tilde\be\in\tcn^\ell$ is another element such that
  $\gr_{\ell}(\al-\tilde\be)\in\ker(\partial)$ then
  $\be-\tilde\be\in\tcn^\ell$ and
  $\gr_\ell(\be-\tilde\be)\in\ker(\partial)$, so
  $\gr_{\ell}(\be-\tilde\be)=0$. Of course, the cohomology class also
  remains unchanged if we add an element of $\Cal N^{\ell+1}$ to
  $\al$.

Hence the element $\gr_{\ell}(\al-\be)$ depends only on the class of
$\al$ in $\gr_\ell(\Cal N/\tcn)$ and we have defined a map
$\gr_\ell(\Cal N/\tcn)\to H^2(\frak m,\gr(\frak g))_\ell$ which is
surjective by part (1). On the other hand, starting from $\al\in\Cal
N^\ell$ the result of our map is zero if and only if there is an
element $\be\in\tcn^\ell$ such that
$\gr_{\ell}(\al-\be)\in\im(\partial)$. By definition of a
normalization condition, this is equivalent to
$\gr_{\ell}(\al-\be)=0$, which exactly means that the class of $\al$
in $\gr_\ell(\Cal N/\tcn)$ is trivial.
\end{proof}

\subsection{Codifferentials}\label{3.3}
A method to obtain a normalization condition in many applications is
via a so--called codifferential. To introduce this concept, we need
some preliminary considerations. Suppose that $(V,\{V^i\})$ and
$(W,\{W^i\})$ are filtered vector spaces and that $\Ph:V\to W$ is a
linear map which is compatible with the filtrations, i.e.~such that
$\Ph(V^i)\subset W^i$ for all $i$.  Then $\Ph$ induces a linear map on
the associated graded vector space, which preserves homogeneities,
i.e.~itself is homogeneous of degree $0$. We denote this map by
$\gr_0(\Ph):\gr(V)\to\gr(W)$ and observe that for each $v\in V^i$ we
get $\gr_0(\Ph)(\gr_i(v))=\gr_i(\Ph(v))\in\gr_i(W)$.

Now $\ker(\Ph)\subset V$ and $\im(\Ph)\subset W$ inherit filtrations,
so we have $\ker(\Ph)^i=\ker(\Ph)\cap V^i$ and
$\gr_i(\ker(\Ph))\subset\gr_i(V)$ and likewise for $\im(\Ph)$. To have
these spaces nicely related to $\gr_0(\Ph)$ an additional technical
condition is needed. 

\begin{definition}\label{def3.3}
  Let $(V,\{V^i\})$ and $(W,\{W^i\})$ be filtered vector spaces and
  let $\Ph:V\to W$ be a linear map which is compatible with the
  filtrations. Then we say that $\Ph$ is \textit{image--homogeneous}
  if and only if for each $i$, and any element $w\in\im(\Ph)\cap W^i$
  there is an element $v\in V^i$ such that $w=\Ph(v)$ or,
  equivalently, iff $\im(\Ph)^i=\Ph(V^i)$ for all $i$.
\end{definition}

\begin{lemma}\label{lem3.3}
  Let $(V,\{V^i\})$ and $(W,\{W^i\})$ be filtered vector spaces and
  let $\Ph:V\to W$ be a linear map which is compatible with the
  filtrations and image--homogeneous. Then for each $i$, the subspaces
  $\gr_i(\ker(\Ph))\subset\gr_i(V)$ and
  $\gr_i(\im(\Ph))\subset\gr_i(W)$ coincide with the kernel and the
  image of $\gr_0(\Ph):\gr_i(V)\to\gr_i(W)$.
\end{lemma}
\begin{proof}
  By definition $v\in\ker(\Ph)^i$ satisfies $v\in V^i$ and $\Ph(v)=0$,
  so $\gr_i(v)\in\gr_i(V)$ satisfies $\gr_0(\Ph)(\gr_i(v))=0$, so
  $\gr_i(\ker(\Ph))\subset \ker(\gr_0(\Ph))$. Conversely, a class in
  $\gr_i(V)$ lies in the kernel of $\gr_0(\Ph)$ if and only if it is
  represented by an element $v\in V^i$ such that $\Ph(v)\in
  W^{i+1}$. By image homogeneity, there exists an element $\tilde v\in
  V^{i+1}$ such that $\Ph(\tilde v)=\Ph(v)$, so $v-\tilde v\in V^i$
  lies in $\ker(\Ph)$ and represents the same class in
  $\gr_i(V)$. This completes the proof for the kernel.

  For the image, the argument is similar. Since for $v\in V^i$ we have
  $\gr_0(\Ph)(\gr_i(v))=\gr_i(\Ph(v))$, we see that the image of
  $\gr_0(\Ph)$ in $\gr_i(W)$ is contained in $\gr_i(\im(\Ph))$. The
  other inclusion follows directly from the definition of an
  image--homogeneous map.
\end{proof}

We will mainly apply this to a map $\Ph$ between spaces of the form
$L(\La^k(\frak g/\frak p),\frak g)$. For such maps, being filtration
preserving just means being compatible with homogeneities of
multilinear maps, so if $\al$ is homogeneous of degree $\geq\ell$,
also $\Ph(\al)$ is homogeneous of degree $\geq\ell$. Moreover, in view
of Lemma \ref{3.1}, the map $\gr_0(\Ph)$ in such a case maps between
the corresponding spaces of the form $C^k(\frak m,\gr(\frak g))$.

\begin{definition}\label{def-codiff}
Let $(\frak g,P)$ be a regular pair. Then a \textit{codifferential}
for $(\frak g,P)$ consists of maps $\partial^*:L(\La^k(\frak
g/\frak p),\frak g)\to L(\La^{k-1}(\frak g/\frak p),\frak g)$ for
$k=2,3$ such that 
\begin{itemize}
\item Both maps are $P$--equivariant, compatible with homogeneities
  and image--homogeneous, and they satisfy $\partial^*\o\partial^*=0$.
\item The induced linear maps $\gr_0(\partial^*):C^k(\frak m,\gr(\frak
  g))\to C^{k-1}(\frak m,\gr(\frak g))$ are \textit{disjoint} to
  $\partial$, in the sense that in $C^k(\frak m,\gr(\frak g))$ we have
  $\ker(\gr_0(\partial^*))\cap \im(\partial)=\{0\}$ for $k=2,3$ and
  $\im(\gr_0(\partial^*))\cap\ker(\partial)=\{0\}$ for $k=1,2$.
\end{itemize}
\end{definition}

\begin{prop}\label{prop3.3}
  Let $(\frak g,P)$ be an admissible pair and suppose that
  $\partial^*$ is a co\-dif\-feren\-tial for $(\frak g,P)$. Then $\Cal
  N:=\ker(\partial^*)\subset L(\La^2(\frak g/\frak p),\frak g)$ is a
  normalization condition for $(\frak g,P)$ and
  $\tcn:=\im(\partial^*)\subset \Cal N$ is a maximal negligible
  submodule.
\end{prop}
\begin{proof}
  Since both maps $\partial^*$ are $P$--equivariant, $\Cal N$ and
  $\tcn$ are $P$--invariant subspaces in $L(\La^2(\frak g/\frak
  p),\frak g)$ and since $\partial^*\o\partial^*=0$, we get
  $\tcn\subset\Cal N$. By Lemma \ref{lem3.3}, we know that, for each
  $\ell>0$, the spaces $\gr_{\ell}(\Cal N)$ and $\gr_\ell(\tcn)$
  coincide with the kernel and the image of $\gr_0(\partial^*)$ in
  $C^2(\frak m,\gr(\frak g))_\ell$, respectively. Now by the
  disjointness assumption, we get $\gr_{\ell}(\Cal
  N)\cap\im(\partial)=\{0\}$ and
  $\gr_{\ell}(\tcn)\cap\ker(\partial)=\{0\}$.

  Disjointness also implies that on $C^2(\frak m,\gr(\frak g))_\ell$,
  we have $\ker(\gr_0(\partial^*)\o\partial)=\ker(\partial)$ and
  $\ker(\partial\o\gr_0(\partial^*))=\ker(\gr_0(\partial^*))$. Together
  these two equations show that, viewed as an endomorphism of
  $\im(\partial)\subset C^2(\frak m,\gr(\frak g))_\ell$, the map
  $\partial\o\gr_0(\partial^*)$ has trivial kernel and thus has to be
  an isomorphism. Hence there is a linear isomorphism $\Ps$ from
  $\im(\partial)$ to itself, which is inverse to
  $\partial\o\gr_0(\partial^*)$ on that subspace. Now writing
$$
\tau=\left(\tau-(\Ps\o \partial\o\gr_0(\partial^*)(\tau)\right)+
\Ps\o \partial\o\gr_0(\partial^*)(\tau),
$$ 
the last summand lies in $\im(\partial)$ and the first part lies in
$\ker(\partial\o\gr_0(\partial^*))=\ker(\gr_0(\partial^*))=\gr_{\ell}(\Cal
N)$. This completes the proof that $\Cal N$ is a normalization
condition.

In the same way, one verifies that $\gr_0(\partial^*)\o\partial$
restricts to a linear isomorphism from
$\im(gr_0(\partial^*))=\gr_\ell(\tcn)$ to itself. Using an inverse,
one shows as above that any element of $C^2(\frak m,\gr(\frak
g))_\ell$ can be written as a sum of an element of $\gr_{\ell}(\tcn)$
and an element of $\ker(\partial)$, which completes the proof.
\end{proof}

\subsection{Examples}\label{3.4} 
\textbf{1. Vanishing prolongation}: In this case, it is often easy to
find direct constructions of normalization conditions or of
codifferentials. What simplifies things is that neither the difference
between filtered and associated graded modules nor the difference
between $P$ and $P/P_+$ plays a role in this situation. We discuss two
different constructions of this type.

Let us first consider the situation related to sub--Riemannian
geometry as discussed in Example 1 of Section \ref{2.6}. So we assume
that $\frak m=\oplus_{i=-\mu}^{-1}\frak m_i$ is a fundamental graded
nilpotent Lie algebra, we fix a positive definite inner product on
$\frak m_{-1}$, and define $\frak p\subset\frak{der}_{\gr}(\frak m)$
to be the Lie algebra of those derivations whose restrictions to
$\frak g_{-1}$ is skew symmetric. Now assume that $P$ is a Lie group
with Lie algebra $\frak p$ which acts on $\frak m$ by automorphisms
preserving the grading and the inner product on $\frak m_{-1}$ in the
obvious sense.

Since $\frak m$ is fundamental, the Lie bracket defines a surjection
$\La^2\frak m_{-1}\to\frak m_{-2}$, which by definition is
$P$--equivariant, and we denote by $\La^2_0\frak m_{-1}$ its
kernel. The $P$--invariant inner product on $\frak m_{-1}$ induces a
$P$--invariant inner product on $\La^2\frak m_{-1}$. The above
surjection can be used to identify $\frak m_{-2}$ as a $P$--module
with the orthocomplement of $\La^2_0\frak m_{-1}$, which induces a
$P$--invariant inner product on $\frak m_{-2}$. Similarly, the bracket
defines a surjection $\frak m_{-1}\otimes\frak m_{-2}\to\frak m_{-3}$
which can be used to define a $P$--invariant inner product on $\frak
m_{-3}$ and so on until we have constructed a $P$--invariant inner
product on all of $\frak m$. Since $\frak p\subset\frak{so}(\frak
m_{-1})$, we can restrict the negative of the Killing form to obtain a
positive definite, $P$--invariant inner product on $\frak p$. Hence we
also obtain an inner product on $\frak g=\frak m\oplus\frak
p\cong\gr(\frak g)$.

Putting these together, we get an induced inner product on each of the
spaces $C^k(\frak m,\gr(\frak g))$, which in this simple situation are
isomorphic to $L(\La^k(\frak g/\frak p),\frak g)$ as $P$--modules. In
this simple case it is also clear that the Lie algebra cohomology
differentials $\partial$ are $P$--equivariant, so we can simply take
their adjoints with respect to the $P$--invariant inner products to
define codifferentials $\partial^*$ in the sense of Definition
\ref{def-codiff}.

Another type of direct construction can be used in the case related to
dual Darboux distributions discussed in Example 1 of Section
\ref{2.6}. Here $\frak m=\frak m_{-2}\oplus\frak m_{-1}$, with $\frak
m_{-1}$ even dimensional and $\frak m_{-2}\subset\La^2\frak m_{-1}$
the kernel of a non--degenerate skew symmetric bilinear form on $\frak
m_{-1}$. Here one defines $P:=\Aut_{\gr}(\frak m)\cong CSp(\frak
m_{-1})$, so this is a reductive group with one--dimensional center
and semisimple part isomorphic to $Sp(\frak m_{-1})$. So again $\frak
g\cong\gr(\frak g)$ and $C^k(\frak m,\gr(\frak g))\cong L(\La^k(\frak
g/\frak p),\frak g)$ as a $P$--module for each $k$. Now $C^2(\frak
m,\gr(\frak g))$ can be analyzed as a representation of the semisimple
part of $P$. The center of $P$ is generated by the grading element,
which acts as $-\id$ on $\frak m_{-1}$, which easily implies that it
acts by a scalar on each subrepresentation of $C^2(\frak m,\gr(\frak
g))$ sitting in fixed homogeneity. This in particular applies to each
irreducible component for the semisimple part. Hence each of these
irreducible components is $P$--invariant.

This readily implies that any $P$--invariant subspace in $C^2(\frak
m,\gr(\frak g))$ admits a $P$--invariant complement. So taking a
$P$--invariant complement to $\im(\partial)$, we obtain a
normalization condition $\Cal N$ for $(\frak g,P)$. Likewise, we can
take an invariant complement $\tcn$ to the $P$--invariant subspace
$\ker(\partial)\cap\Cal N$, to obtain a maximal negligible submodule
in $\Cal N$. It turns out that in the decomposition of $C^2(\frak
m,\gr(\frak g))$ into irreducibles higher multiplicities occur, so
there are several possible choices for $\Cal N$. A concrete example of
a normalization condition is described in the thesis \cite{deZanet}. 

\smallskip

\textbf{2. Parabolics}: For a Lie algebra $\frak g$, there is a
standard complex computing Lie algebra homology with coefficients in a
representation $V$ of $\frak g$. The spaces in this complex are
defined as $C_k(\frak g,V):=\La^k\frak g\otimes V$. The differentials
in the complex, which we denote by $\delta=\delta_V$, lower degree by
one and are explicitly given by
\begin{equation}\label{homology-def} 
\begin{aligned}
\delta(A_1&\wedge\dots\wedge A_k\otimes v):=
\textstyle\sum_{i=1}^k(-1)^iA_1\wedge\dots\wedge\widehat{A_i}\wedge
\dots\wedge A_k\otimes A_i\cdot v\\
&+\textstyle\sum_{i<j}(-1)^{i+j}[A_i,A_j]\wedge
A_1\wedge\dots\wedge\widehat{A_i}\wedge\dots\wedge
\widehat{A_j}\wedge\dots\wedge A_k\otimes v. 
\end{aligned}
\end{equation}
Of course, the spaces $C_k(\frak g,V)$ are naturally $\frak
g$--modules, and from the definition it follows easily that the
differentials $\delta$ are $\frak g$--equivariant. If $\frak
h\subset\frak g$ is a Lie subalgebra, then $V$ is a representation of
$\frak h$ by restriction. By definition, $C_k(\frak h,V)\subset
C_k(\frak g,V)$ and the Lie algebra homology differential for $\frak
h$ coincides with the restrictions of the differential for $\frak g$. 

Now suppose that $\frak g$ is semisimple, $\frak p\leq\frak g$ is a
parabolic subalgebra with nilradical $\frak p_+\subset\frak p$, and
let $P$ be a parabolic subgroup corresponding to $\frak p$ as in
Example 2 of Section \ref{2.6}. There we have noted that $\frak p_+$
is the annihilator of $\frak p$ with respect to the Killing form and
an ideal and a $P$--invariant subspace in $\frak p$. Via the adjoint
action, $\frak g$ is a representation of $\frak p_+$, so as above we
get $C_k(\frak p_+,\frak g)\subset C_k(\frak g,\frak g)$ for each $k$
and the homology differential on this subspace. Via the adjoint
action, these spaces are $P$--submodules, and from the definitions it
follows that the differentials are $P$--equivariant.

Since $\frak p_+$ is the annihilator of $\frak p$, the Killing form
induces a non--degenerate pairing between $\frak g/\frak p$ and $\frak
p_+$ which is compatible with the $P$--actions. Hence for each $k$, we
obtain an isomorphism $L(\La^k(\frak g/\frak p),\frak g)\cong
C_k(\frak p_+,\frak g)$ of $P$--modules. Hence we can view $\delta$ as
a $P$--equivariant map
$$
\partial^*:L(\La^k(\frak g/\frak p),\frak g)\to L(\La^{k-1}(\frak
g/\frak p),\frak g)
$$ 
for each $k$, which in this context usually is called the
\textit{Kostant--codifferential}. It turns out that specializing to
$k=1,2$, we indeed get a codifferential in the sense of Definition
\ref{def-codiff} for any semisimple Lie algebra $\frak g$ and any
parabolic subalgebra $\frak p$, see Section 3.1.11 of \cite{book}. The
resulting normalization conditions are the basis for the proof of
existence of canonical Cartan connections for parabolic geometries in
Section 3.1 of \cite{book}. The fact that the codifferentials for all
parabolic subalgebras are induced by the Lie algebra homology
differential of $\frak g$ is important in the theory of correspondence
spaces and twistor spaces for parabolic geometries, see
\cite{twistor}.

\smallskip

\textbf{3. Algebras related to (systems of) ODEs}: As discussed in
Example 3 of \ref{2.6}, the relevant groups here are $G=(SL(2,\Bbb
R)\times GL(m,\Bbb R))\ltimes V^m_k$ and $P=B\x GL(m,\Bbb R)$. Here
$B\subset SL(2,\Bbb R)$ denotes the Borel subgroup and $V^m_k$ is the
tensor product of the irreducible representation of $SL(2,\Bbb R)$ of
dimension $k+1$ with the standard representation $\Bbb R^m$ of
$GL(m,\Bbb R)$. The normalization condition in this case can also be
expressed by a codifferential. The construction is described in detail
in \cite{CDT} and we just briefly sketch how things work. Let us
denote by $\th$ the standard Cartan involutions $A\mapsto (A^{-1})^*$
on $SL(2,\Bbb R)$, on $GL(m,\Bbb R)$, and on the product of these two
groups. Now on any irreducible representation of either of the
groups, there is a positive definite inner product which is compatible
with the group action in the sense that $\langle A\cdot
v,w\rangle=\langle v,\th(A)\cdot w\rangle$. Doing this for the
representations $V^1_k$ and $\Bbb R^m$, one obtains an induced inner
product on the tensor product $V^m_k$. Together with the inner
products on the Lie algebras of the two groups obtained in the same
way, one gets an inner product on $\frak g$.

This inner product then induces inner products on the spaces
$C^i(\frak g,\frak g)$ of alternating multilinear maps, which naturally
are representations of $G$. Still, these inner products are compatible
with the action of the subgroup $SL(2,\Bbb R)\x GL(m,\Bbb R)$ up to
the action of $\th$. Now for $i=1,2$, consider the Lie algebra
cohomology differential $\partial_{\frak g}:C^i(\frak g,\frak g)\to
C^{i+1}(\frak g,\frak g)$. Since these involve only the Lie bracket on
$\frak g$, they are immediately seen to be $G$--equivariant. Now one
forms the adjoint maps with respect to the inner products defined
above. Simple direct computations show that for $i=1,2$ the adjoint
maps the subspace $L(\La^{i+1}(\frak g/\frak p),\frak g)$ of
horizontal forms to $L(\La^i(\frak g/\frak p),\frak g)$ and that it is
filtration--preserving and $P$--equivariant.

To see that these maps indeed define a codifferential in the sense of
Definition \ref{def-codiff}, one proceeds as follows. The grading of
$\frak g$ defining the filtration is orthogonal with respect to the
inner product constructed above. Hence we can also view this as an
inner product on $\gr(\frak g)$, which can then be restricted to
$\frak m$ and in turn induces inner products on the spaces $C^i(\frak
m,\gr(\frak g))$. Then one easily verifies directly that the maps on
the spaces $C^i(\frak m,\gr(\frak g))$ induced by the adjoints from
above are the adjoints for the Lie algebra cohomology differential
$\partial_{\frak m}$ with respect to the inner product we have just
constructed. This readily implies disjointness while image homogeneity
can be easily verified directly.

\section{Canonical Cartan connections}\label{4} 

In this section we prove the main results of the article, which lead
to existence and uniqueness of canonical Cartan connections. We first
develop the necessary calculus for $\frak g$--valued differential
forms on principal $P$--bundles, and prove basic results on Cartan
connections. Given a normalization condition and a negligible
submodule, we next develop a notion of ``essential curvature''
(generalizing the concept of harmonic curvature) and prove that
essential curvature vanishes if and only if the full curvature
vanishes.

The main technical results are based on the idea of ``normalizing''
Cartan connections and they need only weak assumptions. The first step
works for any admissible pair $(\frak g,P)$, the only requirement is a
normalization condition in the sense of Definition
\ref{def-norm}. Assuming this, we prove that any regular Cartan
connection can be modified to a normal Cartan connection without
losing regularity or changing the underlying filtered
$G_0$--structure. For the second step, we have to assume that $(\frak
g,P)$ is an infinitesimal homogeneous model, i.e.~that $\gr(\frak g)$
is the full prolongation of $(\frak m,\frak g_0)$. Assuming this, we
prove that any principal bundle map between regular normal Cartan
geometries, which induces an isomorphism between the underlying
filtered $G_0$--structures can be modified to an isomorphism of the
Cartan geometries.

To obtain canonical Cartan connections associated to filtered
$G_0$--structures from these results, we need an additional condition
on the algebraic and topological structure of the group
$P$. Basically, this is needed to prove that any principal
$G_0$--bundle can be obtained as a quotient of a principal $P$--bundle
as well as existence of lifts of principal bundle maps.

\subsection{The covariant exterior derivative}\label{4.1} 
Consider an admissible pair $(\frak g,P)$ as in Definition
\ref{def2.2} and a regular Cartan geometry $(p:\Cal G\to M,\om)$ of
type $(\frak g,P)$ as in Section \ref{2.4}. So by definition, $p:\Cal
G\to M$ is a principal $P$--bundle and $\om\in\Om^1(\Cal G,\frak g)$
is a Cartan connection whose curvature has positive homogeneity. As we
have seen in Section \ref{2.4}, this gives rise to a filtration of
$T\Cal G$ by smooth subbundles $T^i\Cal G$ for $i=-\mu,\dots,k$
characterized by $T^i\Cal G=\om^{-1}(\frak g^i)$.

For $k\geq 0$ consider the space $\Om^k(\Cal G,\frak g)$ of $\frak
g$--valued $k$--forms on $\Cal G$. This comes with a natural notion of
homogeneity. We say that $\ph\in\Om^k(\Cal G,\frak g)$ is
\textit{homogeneous of degree} $\geq\ell$ if and only if for tangent
vectors $\xi_j\in T^{i_j}\Cal G$ we always have
$\ph(\xi_1,\dots,\xi_k)\in\frak g^{i_1+\dots+i_k+\ell}$. We will write
$\Om^k(\Cal G,\frak g)^\ell$ for the space of forms which are
homogeneous of degree $\geq\ell$. We will be particularly interested
in forms which are \textit{horizontal} in the sense that they vanish
if one of their entries comes from the vertical subbundle of $p:\Cal
G\to M$, which by definition coincides with $T^0\Cal G$. The space of
horizontal $k$--forms will be denoted by $\Om^k_{\hor}(\Cal G,\frak
g)$ and likewise, we use the notation $\Om^k_{\hor}(\Cal G,\frak
g)^\ell$. Finally, a form $\ph\in\Om^k(\Cal G,\frak g)$ is called
\textit{$P$--equivariant} if for any element $g\in P$ with principal
right action $r^g$ on $\Cal G$, we have
$(r^g)^*\ph=\Ad(g^{-1})\o\ph$. We will denote the space of equivariant
forms by $\Om^k(\Cal G,\frak g)_P$ and combine this in the obvious way
with the other notations we have just introduced.

In this notation, the Cartan connection $\om$ itself is an element of
$\Om^1(\Cal G,\frak g)^0_P$ and its curvature $K$ as introduced in
Section \ref{2.4} lies in $\Om^2_{\hor}(\Cal G,\frak g)_P$. Regularity
of $\om$ by definition is equivalent to the fact that $K$ is
homogeneous of degree $\geq 1$ and thus lies in the subspace
$\Om^2_{\hor}(\Cal G,\frak g)^1_P$. 

\begin{definition}\label{def-covext}
Let $(p:\Cal G\to M,\om)$ be a Cartan geometry of type $(\frak
g,P)$. Then for each $r\geq 0$, we define the \textit{covariant
  exterior derivative} $d^\om:\Om^k(\Cal G,\frak g)\to\Om^{k+1}(\Cal
G,\frak g)$ by defining $d^\om\ph(\xi_0,\dots,\xi_k)$ for vector
fields $\xi_0,\dots,\xi_k\in\frak X(\Cal G)$ as 
\begin{equation}\label{covext}
d\ph(\xi_0,\dots,\xi_k)+\textstyle\sum_{i=0}^k(-1)^i[\om(\xi_i),\ph(\xi_0,\dots,
\widehat{\xi_i},\dots,\xi_k)], 
\end{equation}
where the hat denotes omission and the bracket is in $\frak g$.
\end{definition}

Observe that $d^\om\ph$ evidently is tensorial and since the second
summand in the definition is alternating by construction, it is indeed
an $(k+1)$--form. Let us prove some basic properties of the operation
$d^\om$:

\begin{prop}\label{prop4.1} Let $(p:\Cal G\to M,\om)$ be a regular
  Cartan geometry of type $(\frak g,P)$. Then we have: 

(1) For $\ph\in\Om^k_{\hor}(\Cal G,\frak g)_P$, also $d^\om\ph$ is
  horizontal and $P$--equivariant. In addition, if $\ph$ is
  homogeneous of degree $\geq\ell$, then so is $d^\om\ph$.

  (2) The curvature $K$ of $\om$ satisfies the Bianchi--identity
  $d^\om K=0$.
\end{prop}
\begin{proof}
(1) Equivariancy of $\ph$ reads as $(r^g)^*\ph=\Ad(g^{-1})\o\ph$, and
  as we have noted above, $\om$ is $P$--equivariant, too. Naturality
  of the exterior derivative then shows that
  $(r^g)^*d\ph=d(r^g)^*\ph=d(\Ad(g^{-1})\o\ph)=\Ad(g^{-1})\o
  d\ph$. (In the last step, we have used that we can differentiate
  through the fixed linear map $\Ad(g^{-1})$.) Applying the pullback
  along $r^g$ to the other part in the definition of $d^\om\ph$, we
  get 
$$
\textstyle\sum_i(-1)^i[(r^g)^*\om(\xi_i),(r^g)^*\ph
(\xi_0,\dots,\widehat{\xi_i},\dots,\xi_k)] 
$$
By equivariancy of $\om$ and $\ph$ we can replace $(r^g)^*$ in both
terms by acting with $\Ad(g^{-1})$ on the values. Since $\Ad(g^{-1})$
is a Lie algebra homomorphism, we can move the $\Ad(g^{-1})$ out of
the bracket. Together with the above, we see that equivariancy of
$\ph$ implies equivariancy of $d^\om\ph$. 

Next, we can apply equivariancy of $\ph$ in the case that $g=\exp(tA)$
for some $A\in\frak p$. Then $r^{\exp(tA)}$ is the flow up to time $t$
of the fundamental vector field $\ze_A$ generated by $A$, while
$\Ad(\exp(tA)^{-1})=e^{-t\ad(A)}$ where we use the matrix exponential
in $L(\frak g,\frak g)$. Differentiating at $t=0$, we we obtain $\Cal
L_{\ze_A}\ph=-\ad(A)\o\ph$, where $\Cal L_{\ze_A}$ denotes the Lie
derivative along the fundamental vector field $\ze_A$. Using the
Cartan formula for the Lie derivative and assuming that $\ph$ is
horizontal, we get that $i_{\ze_A}d\ph=-\ad(A)\o\ph$. Since
$\om(\ze_A)=A$, this together with horizontality of $\ph$ implies that
$d^\om\ph(\ze_A,\xi_1,\dots,\xi_k)=0$ for arbitrary vector fields
$\xi_1,\dots,\xi_k$, which exactly says that $d^\om\ph$ is horizontal.

So let us finally assume that $\ph\in\Om^k_{hor}(\Cal G,\frak
g)^\ell_P$. Since we already know that $d^\om\ph$ is horizontal, it
suffices to check homogeneity on sections $\xi_j\in\Ga(T^{i_j}\Cal G)$
with all $i_j<0$. (This uses that any tangent vector in a subbundle
can be extended to a smooth section of the subbundle.) Now we put
$s:=i_0+\dots+i_k+\ell$ and use the global formula for $d\ph$. First,
this gives terms of the form
$\xi_j\cdot\ph(\xi_0,\dots,\widehat{\xi_j},\dots,\xi_k)$. Homogeneity
of $\ph$ shows that the function which is differentiated has values in
$\frak g^{s-i_j}$ and since $i_j<0$, we see that $s-i_j>s$. Thus, also
the derivative along $\xi_j$ has values in $\frak
g^{s-i_j}\subset\frak g^s$.

On the other hand, there are terms in which a bracket of two of the
$\xi_j$ and the remaining vector fields are inserted into $\ph$. In
the proof of Theorem \ref{thm2.4}, we have seen that regularity of
$\om$ implies that $[\xi_j,\xi_r]\in\Ga(T^{i_j+i_r}\Cal G)$ which
readily shows that all these terms have values in $\frak
g^s$. Finally, a term
$[\om(\xi_j),\ph(\xi_0,\dots,\widehat{\xi_j},\dots,\xi_k)]$ clearly
produces values in $[\frak g^{i_j},\frak g^{s-i_j}]\subset\frak g^s$,
and this completes the proof of homogeneity of $d^\om\ph$.

(2) We can compute the value of $d^\om K$ on $\xi,\eta,\ze\in\frak
X(\Cal G)$ as
\begin{equation}\label{Bianchi-comp}
dK(\xi,\eta,\ze)+\textstyle\sum_{cycl}[\om(\xi),K(\eta,\ze)], 
\end{equation}
where $\sum_{cycl}$ denotes the sum over all cyclic permutations of
the arguments. In the first summand, we use $d^2=0$ to replace $K$ by
the two form $(\eta,\ze)\mapsto [\om(\eta),\om(\zeta)]$. Doing this,
the first term can be rewritten as the sum over all cyclic
permutations of 
$$
\xi\cdot[\om(\eta),\om(\ze)]-[\om([\xi,\eta]),\om(\ze)]. 
$$
Using bilinearity of the bracket and applying appropriate cyclic
permutations we can replace this by the sum over all cyclic
permutations of 
$$
[\eta\cdot\om(\ze),\om(\xi)]+[\om(\xi),\ze\cdot\om(\eta)]-
[\om([\eta,\ze]),\om(\xi)]. 
$$
Using skew symmetry in the second summand and inserting the definition
of the exterior derivative, we see that the whole
expression coincides with the cyclic sum over
$[d\om(\eta,\ze),\om(\xi)]$. Using skew symmetry once again and
expanding the definition of $K(\eta,\ze)$ in the second summand of
\eqref{Bianchi-comp} we see that this cancels with the part coming
from $d\om(\eta,\ze)$. Hence we are left with
$\sum_{cycl}[\om(\xi),[\om(\eta),\om(\ze)]]$, which vanishes by the
Jacobi identity for the bracket in $\frak g$. 
\end{proof}

\subsection{The affine structure on the space of Cartan
  connections}\label{4.2} 

Apart from the condition that the values are linear isomorphisms on
each tangent space, Cartan connections form an affine space. Analyzing
the behavior of curvature under affine changes is a major ingredient
in the results of existence and uniqueness on normal Cartan
connections.

\begin{prop}\label{prop4.2}
Let $(\frak g,P)$ be an admissible pair and let $(p:\Cal G\to M,\om)$
be a regular Cartan geometry of type $(\frak g,P)$.

(1) If $\hat\om$ is another Cartan connection on $\Cal G$, then
$\ph:=\hat\om-\om$ lies in $\Om^1_{\hor}(\Cal G,\frak g)_P$. Moreover,
the Cartan connection $\hat\om$ induces the same filtration on $TM$ as
$\om$ iff $\ph\in\Om^1_{\hor}(\Cal G,\frak g)^0_P$ and it induces the
same underlying filtered $G_0$--structure iff $\ph\in\Om^1_{\hor}(\Cal
G,\frak g)^1_P$.

(2) Conversely, for $\ph\in\Om^1_{\hor}(\Cal G,\frak g)_P$, the form
$\hat\om=\om+\ph$ is a Cartan connection on $\Cal G$ iff
$\hat\om(u):T_u\Cal G\to \frak g$ is a linear isomorphism for each
$u\in\Cal G$. This condition is automatically satisfied for
$\ph\in\Om^1_{\hor}(\Cal G,\frak g)^1_P$.

(3) Fix $\ell\geq 1$ and $\ph\in\Om^1_{\hor}(\Cal G,\frak g)^\ell_P$,
put $\hat\om=\om+\ph$ and let $K$ and $\hat K$ be the curvatures of
$\om$ and $\hat\om$, respectively. Then $\hat K-K\in\Om^2_{\hor}(\Cal
G,\frak g)_P^\ell$, so in particular $\hat\om$ is regular, and $\hat
K-K-d^\om\ph\in\Om^2_{\hor}(\Cal G,\frak g)^{\ell+1}_P$. 
\end{prop}
\begin{proof}
(1) By definition, both $\om$ and $\hat\om$ are $P$--equivariant, so
  $\ph$ is equivariant. Since they both reproduce the generators of
  fundamental vector fields, it readily follows that $\ph$ is
  horizontal. The condition that $\hat\om$ induces the same filtration
  of $TM$ of $\om$ clearly means that for $\xi\in T_u\Cal G$ we get
  $\hat\om(\xi)\in\frak g^i$ if and only if  $\om(\xi)\in\frak
  g^i$. But this is evidently equivalent to $\ph(\xi)\in\frak g^i$ for
  all $i<0$ and all $\xi\in T^i_u\Cal G$ and hence to $\ph$ being
  homogeneous of degree $\geq 0$. 

  Assuming that this is the case, consider the construction of the
  underlying filtered $G_0$--structure from Theorem \ref{thm2.4}. The
  homomorphism from $\Cal G$ to the frame bundle of $\gr(TM)$ inducing
  this structure comes from the map on the associated graded induced
  by $\om(u)$, viewed as an isomorphism $T^i_u\Cal G\to\frak g^i$ for
  all $i<0$. So the condition that $\hat\om$ induces the same
  underlying structure means that $\hat\om(u)$ induces the same map on
  the associated graded. But this is equivalent to $\ph(\xi)\in\frak
  g^{i+1}$ for all $\xi\in T^i_u\Cal G$, which completes the proof of
  (1).

  (2) It follows readily from the definition that $\hat\om$ is
  $P$--equivariant and reproduces the generators of fundamental vector
  fields, so the first statement is clear. To prove the second
  statement, assume that $\ph$ is homogeneous of degree $\geq 1$ and
  the $u\in\Cal G$ and $\xi\in T_u\Cal G$ are such that
  $\hat\om(u)(\xi)=0$. Then by definition
  $\om(u)(\xi)=-\ph(u)(\xi)$. But by homogeneity of $\ph$, the right
  hand side lies in $\frak g^{-\mu+1}$, so the left hand side shows
  that $\xi\in T^{-\mu+1}_u\Cal G$. But this implies that the right
  hand side lies in $\frak g^{-\mu+2}$ and thus $\xi\in
  T^{-\mu+2}_u\Cal G$. This can be iterated until we get $\xi\in
  T^0_u\Cal G$. But then $\ph(u)(\xi)=0$ and $\om(u)(\xi)=0$ implies
  $\xi=0$.

(3) Inserting into the definition of curvature, we see that $\hat K$
maps $\xi,\eta\in\frak X(\Cal G)$ to
\begin{equation}\label{curv-trans}
\begin{aligned}
  d\om(\xi,\eta)+d\ph(\xi,\eta)+&[\om(\xi)+\ph(\xi),\om(\eta)+\ph(\eta)]\\
  &=K(\xi,\eta)+d^\om\ph(\xi,\eta)+[\ph(\xi),\ph(\eta)].
\end{aligned}
\end{equation}
From Proposition \ref{prop4.1}, we know that
$d^\om\ph\in\Om^2_{\hor}(\Cal G,\frak g)^\ell_P$. Moreover, for
$\xi\in T^i\Cal G$ and $\eta\in T^j\Cal G$, the last term has values
in $\frak g^{i+j+2\ell}$, so this is homogeneous of degree $\geq
2\ell\geq\ell+1$.
\end{proof}

\subsection{Normal Cartan connections and essential
  curvature}\label{4.3} The key towards the concept of normality and
to normalization is the following description of $\frak g$--valued
differential forms.

\begin{thm}\label{thm4.3}
Let $(\frak g,P)$ be an admissible pair and let  $(p:\Cal G\to M,\om)$
be a regular Cartan geometry of type $(\frak g,P)$. 

(1) For each $k\geq 0$, there is a natural isomorphism between the
spaces $\Om^k(\Cal G,\frak g)$ and the space $C^\infty(\Cal
G,L(\La^k\frak g,\frak g))$. Under this isomorphism, a form is
horizontal iff the corresponding function has values in $L(\La^k(\frak
g/\frak p),\frak g)$ and it in addition is homogeneous of degree
$\geq\ell$ iff the values are in $L(\La^k(\frak g/\frak p),\frak
g)^\ell$. Finally, $P$--equivariancy of a form is equivalent to
equivariancy of the corresponding function $f$ in the sense that for
each $g\in P$, we get $f(u\cdot g)=g^{-1}\cdot f(u)$. Here the
principal right action is used in the right hand side, while in the
left hand side we use the natural action of $P$.

(2) Suppose that $\ph\in\Om^k_{\hor}(\Cal G,\frak g)^\ell_P$
corresponds to $f:\Cal G\to L(\La^k(\frak g/\frak p),\frak
g)^\ell$ under the isomorphism from (1). The the function $\tilde f$
associated to $d^\om\ph\in\Om^{k+1}_{\hor}(\Cal G,\frak g)^\ell_P$ has
the property that 
$$
\gr_\ell\o \tilde f=\partial\o\gr_\ell\o f:\Cal G\to C^{k+1}(\frak
m,\gr(\frak g)). 
$$
\end{thm}
\begin{proof}
  (1) The relation between a form $\ph$ and the corresponding function
  $f$ is given by
\begin{equation}\label{relation}
f(u)(A_1,\dots,A_k):=\ph(u)(\om_u^{-1}(A_1),\dots,\om_u^{-1}(A_k))\in\frak
g 
\end{equation}
for $u\in\Cal G$ and $A_1,\dots,A_k\in\frak g$. For any $A\in\frak g$,
$u\mapsto\om_u^{-1}(A)$ is a smooth vector field on $\Cal G$. Hence given
$\ph\in\Om^k(\Cal G,\frak g)$, the function $f$ defined by
\eqref{relation} has the property that for any choice of elements
$A_i\in\frak g$, the map $u\mapsto f(u)(A_1,\dots,A_k)$ is smooth. But
this means that $f:\Cal G\to L(\La^k\frak g,\frak g)$ is a smooth
map. Conversely, given a smooth function $f$, we define $\ph:\frak
X(\Cal G)^k\to\frak g$ by 
$$
\ph(\xi_1,\dots,\xi_k):=f(\om(\xi_1),\dots,\om(\xi_k))
$$   
for $\xi_i\in\frak X(\Cal G)$. This is obviously alternating and
$k$--linear over $C^\infty(\Cal G,\Bbb R)$ and hence defines an
element of $\Om^k(\Cal G,\frak g)$. Since the two constructions are
evidently inverse to each other we have established the claimed
bijection.

Now a vector field on $\Cal G$ is vertical in $u\in\Cal G$ if and only
if it is mapped by $\om$ to $\frak p\subset\frak g$ in that
point. This shows that a form $\ph$ is horizontal if and only if the
values of the corresponding function $f$ vanish upon insertion of a
single element of $\frak p$. This exactly means that the values lie in
the subspace $L(\La^k(\frak g/\frak p),\frak g)$, which proves the
claimed characterization of horizontality. The interpretation of
homogeneity $\geq\ell$ then follows readily from the definitions.

Equivariancy of $\om$ reads as $\om_{u\cdot
  g}(T_ur^g\cdot\xi)=\Ad(g^{-1})(\om_u(\xi))$ for each $u\in\Cal G$,
$\xi\in T_u\Cal G$ and $g\in P$. This shows that $\om_{u\cdot
  g}^{-1}(A)=T_ur^g\cdot \om_u^{-1}(\Ad(g)(A))$. Inserting this in
\eqref{relation} we conclude that 
$$
f(u\cdot g)(A_1,\dots,A_k)=\ph(u\cdot
g)(T_ur^g\cdot\om_u^{-1}(\Ad(g)(A_1)),\dots,T_ur^g\cdot\om_u^{-1}(\Ad(g)(A_k))), 
$$ 
and the right hand side equals
$(r^g)^*\ph(\om_u^{-1}(\Ad(g)(A_1)),\dots,\om_u^{-1}(\Ad(g)(A_k)))$. Thus
we see that equivariancy of $\ph$ is equivalent to
$$
f(u\cdot
g)(A_1,\dots,A_k)=\Ad(g^{-1})(f(u)(\Ad(g)(A_1),\dots,\Ad(g)(A_k))), 
$$
which exactly means that $f$ is equivariant in the sense claimed in
the theorem. This completes the proof of (1).

(2) Since $\ph\in\Om^k_{\hor}(\Cal G)^\ell_P$, we see from (1) that
the corresponding function $f$ has values in $L(\La^k(\frak g/\frak
p),\frak g)^\ell$, so we can form $\gr_\ell\o f:\Cal G\to C^k(\frak
m,\gr(g))_\ell$. We have also verified in Proposition \ref{prop4.1}
that $d^\om\ph\in\Om^{k+1}_{\hor}(\Cal G,\frak g)^\ell_P$, and as
above we can take the corresponding function $\tilde f$ and form
$\gr_\ell\o \tilde f:\Cal G\to C^{k+1}(\frak m,\gr(\frak g))_\ell$. To
compute this, let us take elements $A_j\in\frak g^{i_j}$ with $i_j<0$
for $j=0,\dots,k$, put $s:=i_0+\dots+i_k+\ell$, and form
$$
\tilde
f(u)(A_0,\dots,A_k)=(d^\om\ph)(u)(\om^{-1}(A_0),\dots,\om^{-1}(A_k)).
$$
Expanding the right hand side according to the definition of $d^\om$,
we get
\begin{equation}\label{dom1}
(d\ph)(u)(\om^{-1}(A_0),\dots,\om^{-1}(A_k))+\textstyle\sum_{j=0}^k
[A_j,f(u)(A_0,\dots,\widehat{A_j},\dots,A_k)]. 
\end{equation}
Now in the last term, we have $A_j\in\frak g^{i_j}$ whereas the value
of $f$ lies in $\frak g^{s-i_j}$. Hence the bracket lies in $\frak
g^s$ and its projection to the associated graded coincides with the
bracket in $\gr(\frak g)$ of $\gr_{i_j}(A_j)$ and
$\gr_{s-i_j}(f(u)(A_0,\dots,\widehat{A_j},\dots,A_k))$. In view of
\eqref{grell-def} we see that $\gr_\ell$ maps the whole last sum in
\eqref{dom1} to
\begin{equation}\label{dom3}
\textstyle\sum_{j=0}^k(-1)^j[\gr_{i_j}(A_j),\gr_{\ell}(f(u))(\gr_{i_0}(A_0)\dots
  ,\widehat{\gr_{i_j}(A_j)},\dots ,\gr_{i_k}(A_k))],  
\end{equation}
with the bracket being taken in $\gr(\frak g)$. 

For the first term in \eqref{dom1}, we have partly analyzed the
exterior derivative in the proof of Proposition \ref{prop4.1}
already. In particular, we have seen there that all the terms in which
values of $\ph$ are differentiated in direction of one of the vector
fields take values in $\frak g^{s+i_j}$ for some $j$ and thus vanish
under projection to the associated graded. Hence we have to compute
the projection of the associated graded of
\begin{equation}\label{dom2}
\textstyle\sum_{j<r}(-1)^{j+r}\ph(u)([\om^{-1}(A_j),\om^{-1}(A_r)],\om^{-1}(A_0),\dots,\widehat{j} ,\dots,\widehat{r},\dots,\om^{-1}(A_k)). 
\end{equation}
In the proof of Theorem \ref{thm2.4} (see in particular formula
\eqref{bracket}), we have seen that regularity of
$\om$ implies that $\om([\om^{-1}(A_j),\om^{-1}(A_r)])$ is congruent
to $[A_j,A_r]$ modulo $\frak g^{i_j+i_r+1}$. Hence up to terms in
$\frak g^{s+1}$, we can compute \eqref{dom2} as 
$$
\textstyle\sum_{j<r}(-1)^{j+r}f(u)([A_j,A_r],A_0,\dots,\widehat{A_j},\dots,\widehat{A_r},\dots,A_k). 
$$ 
Now the degree of the elements inserted into $f$ here add up to
$s-\ell$, so using \eqref{grell-def} once more, we see that the
projection of this into $\gr_s(\frak g)$ is given by 
$$
\textstyle\sum_{j<r}(-1)^{j+r}\gr_\ell(f(u))(\gr_{i_j+i_r}([A_j,A_r]),
\gr_{i_0}(A_0),\dots,\hat j,\dots,\hat r,\dots,\gr_{i_k}(A_k)).  
$$
Now observing that $\gr_{i_j+i_r}([A_j,A_r])$ coincides with the
bracket of $\gr_{i_j}(A_j)$ and $\gr_{i_r}(A_r)$ in $\gr(\frak g)$, we
conclude that this adds up with the contribution from \eqref{dom3} to
$\partial\gr_\ell(f(u))(\gr_{i_0}(A_0),\dots,\gr_{i_k}(A_k))$,
which completes the proof. 
\end{proof}

Having this result at hand, the definition of normality becomes rather
straightforward. Suppose that $(\frak g,P)$ is an admissible pair and
that $(p:\Cal G\to M,\om)$ is a regular Cartan geometry of type
$(\frak g,P)$ with curvature $K\in\Om^2(\Cal G,\frak g)$. Then from
\ref{4.1} we know that $K\in\Om^2_{\hor}(\Cal G,\frak g)^1_P$, so by
Theorem \ref{thm4.3}, it corresponds to a $P$--equivariant smooth
function $\ka:\Cal G\to L(\La^2(\frak g/\frak p),\frak g)^1$. 

\begin{definition}\label{def4.3}
Let $(\frak g,P)$ be an admissible pair and let $\Cal N\subset
L(\La^2(\frak g/\frak p),\frak g)$ be a normalization condition for
$(\frak g,P)$. Let $(p:\Cal G\to M,\om)$ be a regular Cartan geometry
of type $(\frak g,P)$. 

(1) The function $\ka:\Cal G\to L(\La^2(\frak g/\frak p),\frak g)^1$
corresponding to the curvature $K$ of $\om$ is called the
\textit{curvature function} of the geometry. 

(2) The geometry $(p:\Cal G\to M,\om)$ is called \textit{normal} (of
type $\Cal N$) if and only if its curvature function has values in the
subspace $\Cal N\subset L(\La^2(\frak g/\frak p),\frak g)^1$. 

(3) Suppose that the geometry $(p:\Cal G\to M,\om)$ is normal and that
$\tcn\subset\Cal N$ is a negligible submodule. Then the
\textit{essential curvature function} (with respect to
$\tcn\subset\Cal N$) of the geometry is the function $\ka_e:\Cal G\to
\Cal N/\tcn$ induced by $\ka$.
\end{definition}

This definition has several immediate consequences. First, normality
can be checked locally, since it only depends on the curvature, which
is a local quantity. Second, $P$--invariance of the subspace $\Cal N$
shows that if for some $u\in\Cal G$ we have $\ka(u)\in\Cal N$, then
this holds in all points which lie in the same fiber as $u$, since
$\ka(u\cdot g)=g^{-1}\cdot\ka(u)$. This will be crucial for
normalizing Cartan connections. Similarly, the essential curvature
function is a local invariant of a normal geometry. We can easily
prove that the essential curvature still is a complete obstruction
against local flatness.

\begin{prop}\label{prop4.3}
  Let $(\frak g,P)$ be an admissible pair, $\Cal N$ a normalization
  condition for $(\frak g,P)$ and $\tcn\subset\Cal N$ a negligible
  submodule. Then for a regular normal Cartan geometry $(p:\Cal G\to
  M,\om)$, the essential curvature function vanishes identically if and
  only if the curvature $K$ of $\om$ vanishes identically.
\end{prop}
\begin{proof}
  This is a simple consequence of the Bianchi identity. If the
  essential curvature function vanishes identically, then the
  curvature function $\ka$ of the geometry has values in $\tcn\cap
  L(\La^2(\frak g/\frak p),\frak g)^1$, with homogeneity of degree
  $\geq 1$ following from regularity. In particular, for each
  $u\in\Cal G$, we see that $\gr_1(\ka(u))\in\gr_1(\tcn)$. But by the
  Bianchi identity, we have $0=d^\om K$, which using Theorem
  \ref{thm4.3} shows that $0=\partial(\gr_1(\ka(u)))$ for each
  $u\in\Cal G$. By definition of a negligible submodule,
  $\gr_1(\tcn)\cap\ker(\partial)=\{0\}$ so we see that
  $\gr_1(\ka(u))=0$ for each $u\in\Cal G$.

  Hence we conclude that $\ka$ has values in $\tcn\cap L(\La^2(\frak
  g/\frak p),\frak g)^2$, so we can consider $\gr_2(\ka(u))$ for
  $u\in\Cal G$ and show as above that this vanishes. Iteratively, we
  conclude that $\ka(u)$ is homogeneous of degree $2\mu+\nu+1$ for
  each $u\in\Cal G$, which shows that $\ka$ vanishes identically.
\end{proof}

Of course, this result is most interesting in the case of a maximal
negligible submodule $\tcn\subset\Cal N$. Another question of
particular interest is whether $\tcn\subset\Cal N$ can be chosen in
such a way that $P_+\cdot\Cal N\subset\tcn$. This always happens, for
example, for parabolic geometries with $\Cal N$ and $\tcn$ defined via
the Kostant codifferential. If $P_+\cdot\Cal N\subset\tcn$, then $P_+$
acts trivially on $\Cal N/\tcn$, which thus becomes a representation
of $P/P_+=G_0$. The essential curvature function $\ka_e:\Cal G\to \Cal
N/\tcn$ then descends to $\Cal G/P_+$, which is exactly the principal
bundle $\Cal G_0\to M$ describing the underlying filtered
$G_0$--structure. The equivariant function $\ka_e$ then corresponds to
a section of the associated bundle $\Cal G_0\x_{G_0}(\Cal N/\tcn)$, so
this admits a direct interpretation in terms of the underlying
filtered $G_0$--structure. Hence the essential curvature in such a
situation is a much simpler geometric object than the full Cartan
curvature.

\subsection{Normalizing Cartan connections}\label{4.4}
We are now ready to prove the first main result towards the existence
of canonical normal Cartan connections. Namely, we show that a
filtered $G_0$--structure which is induced by some regular Cartan
connection is also induced by a normal regular Cartan connection. This
only requires the nice algebraic properties of a normalization
condition and no additional assumptions on the admissible pair $(\frak
g,P)$. We start with a technical lemma, which should be of independent
interest.

\begin{lemma}\label{lem4.4}
  Let $(\frak g,P)$ be an admissible pair and let $(p:\Cal G\to
  M,\om)$ be a regular Cartan geometry of type $(\frak g,P)$. 

  (1) Suppose that, for some $k$, we have two $P$--invariant subspaces
  $E_1,E_2\subset L(\La^k(\frak g/\frak p),\frak g)$. Then for a
  smooth, $P$--equivariant function $f:\Cal G\to L(\La^k(\frak g/\frak
  p),\frak g)$, which has values in $E_1+E_2$, there are smooth,
  $P$--equivariant functions $f_j:\Cal G\to E_j$ for $j=1,2$ such that
  $f=f_2+f_2$.

  (2) Suppose that, for some $k$ and $\ell$, we have a smooth,
  $P$--equivariant function $f:\Cal G\to \im(\partial)\cap C^k(\frak
  m,\gr(\frak g))_\ell$. Then there is a smooth, $P$--equivariant
  function $h:\Cal G\to L(\La^{k-1}(\frak g,\frak p),\frak g)^\ell$
  such that $f=\partial\o\gr_\ell\o h$.
\end{lemma}
\begin{proof}
  For both parts, we first solve the problem locally and then glue to
  a global solution using a partition of unity. Hence we choose an
  open covering $\{U_i:i\in I\}$ of $M$ such that $\Cal G$ is trivial
  over each $U_i$. We fix a location section $\si_i$ of $\Cal G$ over
  each $U_i$ and choose a partition of unity subordinate to the
  covering $\{U_i\}$, which we denote by $\{\alpha_i:i\in I\}$. 

  (1) Choose a linear subspace $W\subset E_1$ which is complementary
  to $E_2$ in the finite dimensional vector space $E_1+E_2$. For each
  $i$, $f\o\si_i$ is a smooth function $U_i\to E_1+E_2$, and thus can
  be uniquely written as $h_1+h_2$, where $h_1$ has values in
  $W\subset E_1$, $h_2$ has values in $E_2$, and both summands are
  smooth. Since $\Cal G|_{U_i}\cong U_i\x P$, there is a unique
  $P$--equivariant smooth function $f^i_1:p^{-1}(U_i)\to E_1$ such
  that $h_1=f^i_1\o\si_i$. (One simply puts $f^i_1(\si_i(x)\cdot
  g)=g^{-1}\cdot h_1(x)$, which has values in the $P$--invariant
  subspace $E_1$.) 

  In the same way, we find a $P$--equivariant smooth function
  $f_2^i:p^{-1}(U_i)\to E_2$ such that $h_2=f^i_2\o\si_i$. Hence by
  construction we have $f\o\si_i=(f_1^i+f_2^i)\o\si_i$ which by
  equivariancy implies that $f|_{p^{-1}(U_i)}=f_1^i+f_2^i$. Now for
  each $i$ and $j=1,2$, $(\al_i\o p)f^i_j$ is a smooth,
  $P$--equivariant function on $p^{-1}(U_i)$ which can be extended by
  zero to all of $\Cal G$. Defining $f_j:=\sum_{i\in I}(\al_i\o
  p)f^i_j$, we obtain $P$--equivariant smooth functions $\Cal G\to
  E_j$ for $j=1,2$, which clearly satisfy $f=f_1+f_2$.

  (2) The composition $\partial\o \gr_\ell$ defines a surjection from
  $L(\La^{k-1}(\frak g/\frak p),\frak g)^\ell$ onto the subspace
  $\im(\partial)\subset C^k(\frak m,\gr(\frak g))_\ell$. Since these
  are finite dimensional vector spaces, we can choose a linear right
  inverse $\psi:\im(\partial)\to L(\La^{k-1}(\frak g/\frak p),\frak
  g)^\ell$ to this map. Since $f$ has values in $\im(\partial)$ we
  can, for each $i\in I$, consider the smooth map $\ps\o f\o
  \si_i:U_i\to L(\La^{k-1}(\frak g/\frak p),\frak g)^\ell$. As in the
  proof of part (1), this can be uniquely written as $h^i\o\si_i$ for
  a smooth, $P$--equivariant function $h^i:p^{-1}(U_i)\to
  L(\La^{k-1}(\frak g/\frak p),\frak g)^\ell$. By construction, we
  have $\partial\o \gr_\ell\o h^i=f$ along the image of $\si_i$ and
  hence on all of $p^{-1}(U_i)$ by equivariancy. As in part (1),
  $h=\sum_{i\in I}(\al_i\o p)h^i$ does the job.
\end{proof}

Having this at hand, we can prove the main result on normalizing
regular Cartan connections.

\begin{thm}\label{thm4.4}
  Let $(\frak g,P)$ and admissible pair and let $\Cal N$ be a
  normalization condition for $(\frak g,P)$. Let $(p:\Cal G\to M,\om)$
  be a regular Cartan geometry of type $(\frak g,P)$. Then there is a
  regular normal Cartan connection $\tilde\om$ on $\Cal G$, which
  induces the same underlying filtered $G_0$--structure (in the sense
  of Theorem \ref{thm2.4}) as $\om$.
\end{thm}
\begin{proof}
  We prove this via the following iterative construction. Suppose that
  $\om$ has the property that its curvature function $\ka$ has values
  in $\Cal N+L(\La^2(\frak g/\frak p),\frak g)^\ell$ for some
  $\ell\geq 1$. For $\ph\in\Om^1_{\hor}(\Cal G,\frak g)^\ell_P$ we
  then know from Proposition \ref{prop4.2} that $\hat\om=\om+\ph$ is a
  regular Cartan connection inducing the same underlying filtered
  $G_0$--structure as $\om$, and we construct $\ph$ in such a way that
  the curvature function $\hat\ka$ of $\hat\om$ has values in $\Cal
  N+L(\La^2(\frak g/\frak p),\frak g)^{\ell+1}$. Since the initial
  assumption is trivially satisfied for the initial Cartan connection
  $\om$ and $\ell=1$, we can iteratively apply this construction until
  we get a curvature function with values in $\Cal N+L(\La^2(\frak
  g/\frak p),\frak g)^{2\mu+\nu+1}=\Cal N$.

  So assume that the curvature function $\ka:\Cal G\to L(\La^2(\frak
  g/\frak p),\frak g)^1$ of $\om$ has values in $\Cal N+L(\La^2(\frak
  g/\frak p),\frak g)^\ell$. By part (1) of Lemma \ref{lem4.4}, we can
  write $\ka$ as a sum $\ka=\ka_1+\ka_2$ of two smooth,
  $P$--equivariant functions such that $\ka_1$ has values in $\Cal N$
  and $\ka_2$ has values in $L(\La^2(\frak g/\frak p),\frak
  g)^\ell$. Now consider the composition $\gr_\ell\o\ka_2:\Cal G\to
  C^2(\frak m,\gr(\frak g))_\ell$. By definition of a normalization
  condition, the target space splits into the direct sum of
  $\gr_\ell(\Cal N)$ and (its intersection with)
  $\im(\partial)$. Applying part (2) of Lemma \ref{lem4.4} to the
  negative of the $\im(\partial)$--component of $\gr_\ell\o\ka_2$, we
  obtain a $P$--equivariant smooth function $h:\Cal G\to L(\frak
  g/\frak p,\frak g)^\ell$. By construction this has the property that
  $\gr_\ell\o\ka_2+\partial\o\gr_\ell\o h$ has values in
  $\gr_{\ell}(\Cal N)\subset C^2(\frak m,\gr(\frak g))_\ell$.

  Now take the form $\ph\in\Om^1_{\hor}(\Cal G,\frak g)^\ell_P$
  corresponding to this function $h$ and put $\hat\om=\om+\ph$. By
  Proposition \ref{prop4.2}, the curvature $\hat K$ of $\hat\om$
  coincides with $K+d^\om\ph$ up to terms of homogeneity
  $\geq\ell+1$. Equivalently, the function which maps $A,B\in\frak g$
  to $\hat K(\om^{-1}(A),\om^{-1}(B))$ can be written, up to terms
  which are homogeneous of degree $\geq\ell+1$, as
$$
\ka_1(A,B)+\ka_2(A,B)+d^\om\ph(\om^{-1}(A),\om^{-1}(B)).
$$ 
Now the first summand has values in $\Cal N$, while the rest is
homogeneous of degree $\geq\ell$. Applying $\gr_\ell$ to that part, we
see from Theorem \ref{thm4.3} that we get
$\gr_\ell\o\ka_2+\partial\o\gr_\ell\o h$ which by construction lies in
$\gr_\ell(\Cal N)$. But this exactly means that the second part has
values in $\Cal N+L(\La^2(\frak g/\frak p),\frak g)^{\ell+1}$. Thus
also $(A,B)\mapsto \hat K(\om^{-1}(A),\om^{-1}(B))$ has values in that
subspace. But now $\hat\om(\om^{-1}(A))=\ph(A)$, so
$\hat\om^{-1}(A)=\om^{-1}(A)-\hat\om^{-1}(\ph(A))$ and hence the curvature
function $\hat\ka$ maps $(A,B)$ to
\begin{align*}
\hat K(\om^{-1}(A),&\om^{-1}(B))-\hat
K(\hat\om^{-1}(\ph(A)),\om^{-1}(B))\\&-K(\om^{-1}(A),\hat\om^{-1}(\ph(B)))+\hat
K(\hat\om^{-1}(\ph(A)),\hat\om^{-1}(\ph(B))).
\end{align*}
But for $A\in\frak g^i$ and $B\in\frak g^j$, the second and third term
have values in $\frak g^{i+\ell+j+1}$ and the last term even has
values in $\frak g^{i+j+2\ell+1}$. Hence we conclude that $\hat\ka$
has values in $\Cal N+L(\La^2(\frak g/\frak p),\frak g)^{\ell+1}$,
which completes the proof.
\end{proof}

\subsection{Uniqueness of normal Cartan connections}\label{4.5}
To prepare for the proof of uniqueness, consider an admissible pair
$(\frak g,P)$ and a regular Cartan geometry $(p:\Cal G\to M,\om)$ of
type $(\frak g,P)$. Suppose that $\Ph:\Cal G\to\Cal G$ is a smooth
homomorphism of principal bundles which covers the identity on
$M$. Then for each $u\in\Cal G$, the point $\Ph(u)$ lies in the same
fiber of $\Cal G$ as $u$, so there is an element $g(u)\in P$ such that
$\Ph(u)=u\cdot g(u)$, and smoothness of $\Ph$ implies smoothness of
$g:\Cal G\to P$. Moreover, since $\Ph(u\cdot h)=\Ph(u)\cdot h$ for all
$h\in P$, we must have $g(u\cdot h)=h^{-1}g(u)h$ for all $u\in\Cal G$
and $h\in P$. Conversely, if we assume that $g:\Cal G\to P$ is a
smooth map such that $g(u\cdot h)=h^{-1}g(u)h$, then $\Ph(u)=u\cdot
g(u)$ defines an automorphism of $\Cal G$ which covers the identity on
$M$.

A simple way how to construct such functions is via the Lie
algebra. Suppose that $Z:\Cal G\to \frak p=\frak g^0$ is a smooth
function such that for each $u\in\Cal G$ and $h\in P$ we get $Z(u\cdot
h)=\Ad(h^{-1})(Z(u))$. Then $g(u):=\exp(Z(u))$ defines a smooth
function $\Cal G\to P$ such that $g(u\cdot h)=h^{-1}g(u)h$ for all
$u\in\Cal G$ and $h\in P$. Apart from the linear structure, this also
has the advantage that we can require $Z$ to have values in one of the
subspaces $\frak g^i$ with $i>0$.

Next, observe that an automorphism $\Ph:\Cal G\to\Cal G$ of the
principal bundle $\Cal G$ is a diffeomorphism and satisfies $\Ph\o
r^g=r^g\o\Ph$ for all $g\in P$. This easily implies that for any such
automorphism, the pullback $\Ph^*\om$ of $\om$ is again a Cartan
connection of type $(\frak g,P)$. Now for automorphisms of the special
form constructed above, we can describe the relation between $\om$ and
$\Ph^*\om$.

\begin{lemma}\label{lemma4.5}
  Let $(\frak g,P)$ be an admissible pair and let $(p:\Cal G\to
  M,\om)$ be a Cartan geometry of type $(\frak g,P)$. For some
  $\ell>0$ let $Z:\Cal G\to \frak g^{\ell}$ be a smooth map, consider
  the principal bundle automorphism $\Ph:\Cal G\to\Cal G$ defined by
  $\Ph(u):=u\cdot\exp(Z(u))$ and the pullback $\Ph^*\om$ of $\om$. 

  Then for each $u\in \Cal G$, $i<0$, and $\xi\in T_u^i\Cal
  G$ we get $\Ph^*\om(u)(\xi)-\om(u)(\xi)\in\frak g^{i+\ell}$ and the
  class of this element in $\gr_{i+\ell}(\frak g)$ coincides with
  $-[\gr_\ell(Z(u)),\gr_i(\om(u)(\xi))]$.
\end{lemma}
\begin{proof}
  As before, put $g(u)=\exp(Z(u))$ for all $u\in\Cal G$. Let $r:\Cal
  G\x P\to\Cal G$ be the principal right action and for $u\in\Cal G$
  and $h\in P$ consider the corresponding partial maps $r_u:P\to\Cal
  G$ and $r^h:\Cal G\to\Cal G$ defined by $r_u(h)=r^h(u)=u\cdot h$. By
  definition $\Ph=r\o(\id,g)$, so for $u\in\Cal G$ and $\xi\in T_u\Cal
  G$, we obtain
$$
T_u\Ph\cdot\xi=T_{(u,g(u))}r\cdot
(\xi,T_ug\cdot\xi)=T_ur^{g(u)}\cdot\xi+ T_{g(u)}r_u\cdot T_ug\cdot\xi,
$$ 
where in the last equality we have used that
$(\xi,T_ug\cdot\xi)=(\xi,0)+(0,T_ug\cdot\xi)$. The second summand in
this expression can be computed explicitly, compare with the proof of
Proposition 3.1.14 in \cite{book}, but for our purposes, a rough
description is sufficient. Since for $\ell>0$, the filtration
component $\frak g^\ell$ is a Lie subalgebra in $\frak g^0=\frak p$,
it generates a connected virtual Lie subgroup of $P$. By construction,
$g$ has values in this subgroup, which implies that for any $\xi$ the
tangent vector $T_ug\cdot\xi\in T_{g(u)}P$ can be realized
$\tfrac{d}{dt}|_{t=0}g(u)\cdot\exp(tW)$ for some $W\in\frak
g^\ell$. Acting by $T_{g(u)}r_u$ on that tangent vector, we get
$$
\tfrac{d}{dt}|_{t=0}u\cdot (g(u)\exp(tW))=\ze_W(u\cdot g(u))\in
T^\ell_{\Ph(u)}\Cal G. 
$$ 
This is mapped by $\om$ to $\frak g^\ell$, which for each $i<0$ is
contained in $\frak g^{i+\ell+1}$. Hence for $\xi\in T^i\Cal G$ with
$i<0$ we can compute $(\Ph^*\om)(u)(\xi)=\om(\Ph(u))(T_u\Ph\cdot\xi)$
up to terms in $\frak g^{i+\ell+1}$ as
$$
\om(u\cdot g(u))(T_ur^{g(u)}\cdot\xi)=\Ad(g(u)^{-1})(\om(u)(\xi)). 
$$ Since $g(u)=\exp(Z(u))$ we get $\Ad(g(u)^{-1})=e^{-\ad(Z(u))}$. By
assumption, we have $\om(u)(\xi)\in\frak g^i$, so since $Z\in\frak
g^\ell$ we get $\ad(Z(u))^2(\frak g^i)\subset\frak
g^{i+2\ell}\subset\frak g^{i+\ell+1}$. Hence ignoring terms in $\frak
g^{i+\ell+1}$, we can replace $e^{-\ad(Z(u))}$ by $(\id-\ad(Z(u)))$
which readily implies all claims of the lemma.
\end{proof}

Using this, we can now prove the basic result on uniqueness. 

\begin{thm}\label{thm4.5}
  Let $(\frak g,P)$ be an infinitesimal homogeneous model for filtered
  $G_0$--structures and let $\Cal N$ be a normalization condition for
  $(\frak g,P)$. Let $(p:\Cal G\to M,\om)$ be a regular normal Cartan
  geometry of type $(\frak g,P)$ and let $\hat\om$ be another normal
  Cartan connection on $\Cal G$, which induces the same underlying
  filtered $G_0$--structure as $\om$. Then there is an automorphism
  $\Ph$ of the principal $P$--bundle $\Cal G$ which induces the
  identity on the underlying $G_0$--bundle $\Cal G/P_+$ such that
  $\Ph^*\hat\om=\om$.
\end{thm}
\begin{proof}
  From Proposition \ref{prop4.2}, we know that $\hat\om$ induces the
  same underlying filtered G--structure as $\om$ iff the difference
  $\hat\om-\om$ lies in $\Om^1_{\hor}(\Cal G,\frak g)^1_P$. We prove
  the theorem by a recursive construction. Assuming that
  $\hat\om-\om\in \Om^1_{\hor}(\Cal G,\frak g)^\ell_P$ for some
  $\ell\geq 1$, we construct an automorphism $\Ph$ of $\Cal G$
  inducing the identity on $\Cal G/P_+$ such that $\Ph^*\hat\om-\om\in
  \Om^1_{\hor}(\Cal G,\frak g)^{\ell+1}_P$. Since $\Ph$ induces the
  identity on $\Cal G/P_+$, the Cartan connection $\Ph^*\hat\om$
  induces the same underlying filtered $G_0$--structure as $\hat\om$
  and hence as $\om$. Thus we can iterate the argument, until we
  arrive at $\Ph^*\hat\om-\om\in \Om^1_{\hor}(\Cal G,\frak
  g)^{\mu+\nu+1}_P$, which is the zero space by homogeneity.

So let us assume that $\ph:=\hat\om-\om\in \Om^1_{\hor}(\Cal G,\frak
g)^\ell_P$ for some $\ell\geq 1$ and denote by $f:\Cal G\to L(\frak
g/\frak p,\frak g)^\ell$ the corresponding $P$--equivariant
function. By Proposition \ref{prop4.3}, the curvatures $K$ and $\hat
K$ have the properties that $\hat K-K-d^{\om}\ph$ is homogeneous of
degree $\geq\ell+1$. We have also seen in the Proof of Theorem
\ref{thm4.4} that the curvature function $\hat\ka$ differs from the
function $(A,B)\mapsto\hat K(\om^{-1}(A),\om^{-1}(B))$ by a function
which is homogeneous of degree $\geq\ell+1$. Using this and Theorem
\ref{thm4.3}, we conclude that $\hat\ka-\ka$ has values in
$L(\La^2(\frak g/\frak p),\frak g)^\ell$ and that
$\gr_\ell\o(\hat\ka-\ka)=\partial\o\gr_\ell\o f$. 

Now since both $\hat\om$ and $\om$ are normal, we see that
$\hat\ka-\ka$ has values in $\Cal N$. But by definition of a
normalization condition, $\gr_\ell(\Cal N)\cap\im(\partial)=\{0\}$, so
we conclude that $\gr_\ell\o f$ has values in $\ker(\partial)\subset
L(\frak m,\gr(\frak g))_\ell$. Since $\gr(\frak g)$ is the full
prolongation of $(\frak m,\gr_0(\frak g))$, this space coincides with
$\partial(\gr_\ell(\frak g))$. If $\ell>\nu$, then we directly get
$\gr_\ell\o f=0$, so $\ph$ actually is homogeneous of degree $\ell+1$
and iterating the argument, we conclude that $\hat\om=\om$ in this
case. 

If $\ell<\nu$, then applying part (2) of Lemma \ref{lem4.4} to
$-\gr_\ell\o f$, we obtain a smooth, $P$--equivariant function $Z:\Cal
G\to\frak g^\ell$ such that $\gr_\ell\o f=-\partial\o\gr_\ell\o Z$.
Now we define an automorphism $\Ph$ of $\Cal G$ as
$\Ph(u):=u\cdot\exp(Z(u))$ and form the pullback $\Ph^*\hat\om$. By
Lemma \ref{lemma4.5}, the difference $\Ph^*\hat\om-\hat\om$ is
homogeneous of degree $\geq\ell$ so the same holds for
$\Ph^*\hat\om-\om$. Denoting by $\tilde f$ the function corresponding
to $\Ph^*\hat\om-\hat\om$, Lemma \ref{lemma4.5} shows that
$\gr_\ell\o\tilde f=\partial\o\gr_\ell\o Z=-\gr_\ell\o f$. But this
exactly says that the composition of $\gr_\ell$ with the function
corresponding to $\Ph^*\hat\om-\om$ vanishes identically, so
$\Ph^*\hat\om-\om$ is homogeneous of degree $\geq\ell+1$, and this
completes the proof.
\end{proof}

\subsection{Canonical Cartan connections}\label{4.6} 
As our final result, we show that under an additional condition on
the group $P$, we get an equivalence of categories between filtered
$G_0$--structures and regular normal Cartan geometries. This
condition is satisfied in most of the examples that I am aware of.

\begin{definition}\label{def4.6}
  Let $(\frak g,P)$ be an admissible pair, let $P_+\subset P$ be the
  subgroup introduced in Section \ref{2.3} and put $G_0:=P/P_+$. Then
  we say that \textit{$P$ is of split exponential type} if there is a
  smooth homomorphism $\iota:G_0\to P$ such that the map $G_0\x \frak
  g^1\to P$  defined by $(g_0,Z)\mapsto \iota(g_0)\exp(Z)$ is a global
  diffeomorphism. 
\end{definition}

Observe that this in particular implies that $g_0=\iota(g_0)P_+$ for
each $g_0\in G_0$, so $\iota$ splits the quotient projection $P\to
G_0$. The splitting of this quotient projection is the main
requirement imposed by the condition. This follows since $\frak g^1$
is nilpotent by definition, so the exponential mapping always is a
diffeomorphism from $\frak g^1$ onto the universal covering of the
connected component of the identity of $P_+$.

Notice also that $P$ is of split exponential type for the examples
discussed in Section \ref{2.6}. In the case of vanishing prolongation,
we have $P=G_0$ and for parabolics this is a well known fact, see
Theorem 3.1.3 of \cite{book}, which also handles the case related to
(systems of) ODEs. Finally, also the general construction discussed in
Remark \ref{rem2.6} always produces a group $P$ which is of split
exponential type.

\begin{thm}\label{thm4.6}
  Let $(\frak g,P)$ be an infinitesimal homogeneous model for filtered
  $G_0$--structures such that $P$ is of split exponential type. Let
  $\Cal N$ be a normalization condition for $(\frak g,P)$. Then the
  category of regular normal Cartan geometries of type $(\frak g,P)$
  is equivalent to the category of filtered $G_0$--structures. More
  explicitly, we have

(1) Any filtered $G_0$--structure can be realized as the underlying
  structure of a regular normal Cartan geometry of type $(\frak g,P)$,
  which is uniquely determined up to isomorphism.
 
(2) For two regular normal Cartan geometries of type $(\frak g,P)$,
  any morphism between the underlying filtered $G_0$--structures lifts
  to a morphism of Cartan geometries. 
\end{thm}
\begin{proof}
  (1) Let $(M,\{T^iM\})$ be a filtered manifold which is regular of
  type $\frak m$ and let $p_0:\Cal G_0\to M$ define a filtered
  $G_0$--structure. This means that $\Cal G_0$ is a principal bundle
  with structure group $G_0$ and comes with a homomorphism to the
  frame bundle of $\gr(TM)$ which covers the identity on $M$. Hence to
  each point $u_0\in\Cal G$ we can associate a family of linear
  isomorphisms $\phi_i(u):\gr_i(T_{p_0(u)}M)\to\gr_i(\frak g)$ for
  $i=-\mu,\dots,-1$, which depend smoothly on $u$ in a way compatible
  with the $G_0$--actions.

Next, let $\iota:G_0\to P$ be a smooth homomorphism as in Definition
\ref{def4.6}. Via $\iota$, $G_0$ acts on $P$ by left multiplication
and we take the associated bundle $p:\Cal G:=\Cal G_0\x_{G_0}P\to
M$. This is well known to be a principal $P$--bundle and the first
projection induces a well defined smooth map $q:\Cal G\to\Cal G_0$
which descends to an isomorphism $\Cal G/P_+\to\Cal G_0$. Thus we have
realized $\Cal G_0$ globally as a quotient of a principal
$P$--bundle. For later use, we choose a principal connection $\ga$ on
$p:\Cal G\to M$ and for $u\in\Cal G$, we denote by $H_u$ the
horizontal subspace $\ker(\ga(u))$. We also fix a splitting
$j:\gr(\frak g)\to\frak g$ of the filtration of $\frak g$ as in the
proof of Lemma \ref{lem3.1}.

Suppose that $U\subset M$ is an open subset over which the bundle
$\Cal G_0$ is trivial. Then also $\Cal G$ is trivial over $U$ and we
claim that there is a regular Cartan connection $\om_U$ on
$p^{-1}(U)\subset\Cal G$ which induces the filtered $G_0$--structure
$(p_0)^{-1}(U)\to U$.

To see this, take a smooth section $\si:U\to\Cal G$ and the induced
section $\si_0:=q\o\si:U\to \Cal G_0$. For $x\in U$, the point
$\si_0(x)\in\Cal G_0$ determines linear isomorphisms $\gr_i(T_xM)\to
\gr_i(\frak g)$ for each $i=-\mu,\dots,-1$. Fixing bases of the spaces
$\gr_i(\frak g)$, this gives rise to smooth local frames of the
bundles $\gr_i(TM)$ over $U$ all $i=-\mu,\dots,-1$. Now for each $i$,
we can choose sections of $T^iM$ lifting the elements of that
frame. From this construction, it follows readily that the resulting
sections for all $i$ together are linearly independent in each point
and thus form a frame for $TM$ defined over $U$.

For $x\in U$, we define a map $\ps_x:T_xM\to \frak g$ by requiring
that for an element $\xi$ in the frame for $TM$ corresponding to a
basis element $X\in\gr_i(\frak g)$ the tangent vector $\xi(x)$ is
mapped to $j(X)\in\frak g^i$. By construction, composing the
projection $\frak g\to\frak g/\frak p$ with $\ps_x$, one obtains a
linear isomorphism, so in particular $\ps_x$ is injective. It also
follows readily that $\ps_x(T^i_xM)\subset\frak g^i$ for each
$i=-\mu,\dots,-1$ and that for a smooth vector field $\xi$ on $U$, the
map $U\to\frak g$ defined by $x\mapsto \ps_x(\xi)$ is smooth, too.
Now for each $x\in U$, we define $\tilde\ps_x:T_{\si(x)}\Cal G\to
\frak g$ by
$$
\tilde\ps_x(\xi):=\ps_x(T_{\si(x)}p\cdot\xi)+\ga(\xi)(\si(x)).
$$ 
We first observe that $\tilde\ps_x(\xi)\in\frak p$ if and only if
$T_{\si(x)}p\cdot\xi=0$ and hence $\xi$ is vertical. But on
vertical vector fields, $\ga$ is injective, so $\tilde\ps_x$ is
injective and thus a linear isomorphism. It also follows readily from
the definition that $\tilde\ps_x(\ze_A(\si(x)))=A$ for all $A\in\frak
p$. Finally, it is clear by construction that for a smooth vector
field $\xi$ on $p^{-1}(U)$, the map $U\to\frak g$ defined by $x\mapsto
\ps_x(\xi(\si(x)))$ is smooth. Now observe that
$(x,g)\mapsto\si(x)\cdot g$ defines a global diffeomorphism $U\x P\to
p^{-1}(U)\subset\Cal G$. This easily implies that
$$
\om(\si(y)\cdot
g)(\xi):=\Ad(g^{-1})(\tilde\ps_y(T_{\si(y)\cdot g}r^{g^{-1}}\cdot\xi)),
$$ 
defines a form $\om\in\Om^1(p^{-1}(U),\frak g)_P$ which is uniquely
characterized by equivariancy and the fact that
$\om(\si(x))=\tilde\ps_x$ for all $x\in U$. It readily follows that
the values of $\om$ all are linear isomorphisms and equivariancy of
fundamental vector fields implies that $\om(\ze_A)=A$ for all
$A\in\frak p$, so $\om$ is a Cartan connection on $p^{-1}(U)$. 

Along $\si(U)$, $\om$ by construction has the property that
$\om(\xi)\in\frak g^i$ for some $i<0$ if and only if $Tp\cdot\xi\in
T^iM$. By equivariancy of $\om$, this remains true on all of
$p^{-1}(U)$, so $\om$ induces the given filtration $\{T^iM\}$ of
$TM$. Moreover, if
$\om(\si(x))(\xi)=\tilde\ps_x(T_{\si(x)}p\cdot\xi)\in\frak g^i$ then
the class of this element in $\gr_i(\frak g)$ is obtained by taking
the class of $T_{\si(x)}p\cdot\xi$ in $\gr_i(T_xM)$ and mapping it to
$\gr_i(\frak g)$ via the linear isomorphism corresponding to
$\si_0(x)=q(\si(x))\in\Cal G_0$. But since
$T_{\si(x)}p=T_{\si_0(x)}p_0\o T_{\si(x)}q$ we conclude that
$\om(\si(x))$ induces the isomorphisms determined by $\si_0(x)$ via
the construction from Theorem \ref{thm2.4}. Since $q\o r^g=r^{gP_+}\o
q$ for all $g\in P$, equivariancy implies that $\om(u)$ induces the
isomorphisms corresponding to $q(u)\in \Cal G_0$ for all $u\in
p^{-1}(U)$, so $\om$ is regular and induces the given filtered
$G_0$--structure. This completes the proof of the claim.

By topological dimension theory (see Section 1.2 in \cite{GHV}), the
bundle $\Cal G_0$ admits a finite atlas, so we can find (possibly
disconnected) open subsets $U_i\subset M$ for $i=1,\dots,N$ such that
$M=U_1\cup\dots\cup U_N$ and such that $\Cal G_0$ is trivial over each
$U_i$. Since $M$ is a normal topological space, we further find open
subsets $V_i\subset M$ such that $\overline{V_i}\subset U_i$ for each
$i$ and such that $M=V_1\cup\dots\cup V_N$. By the last step, we can
find a regular Cartan connection $\om_i\in\Om^1(p^{-1}(U_i),\frak g)$
inducing the given filtered $G_0$--structure on $M$ for each $i$. Now
over $U_{12}:=U_1\cap U_2$ we have the Cartan connections $\om_1$ and
$\om_2$ which induce the same underlying filtered G--structure, so
$\phi:=\om_2|_{p^{-1}(U_{12})}-\om_1|_{p^{-1}(U_{12})}\in\Om^1_{\hor}(p^{-1}(U_{12}))^1_P$
by Proposition \ref{prop4.2}. Now choose a bump function $f:M\to
[0,1]$ with support contained in $U_2$, which is identically one on
$V_2$. Then $u\mapsto f(p(u))\ph(u)$ extends by zero to an element of
$\Om^1_{\hor}(p^{-1}(U_1))^1_P$, so $\om_1(u)+f(p(u))\ph(u)$ defines a
regular Cartan connection on $p^{-1}(U_1)$ which induces the given
filtered $G_0$--structure on $U_1$. But on $p^{-1}(U_1\cap V_2)$, this
Cartan connection by construction coincides with $\om_2$, so together
these two forms define a Cartan connection on $p^{-1}(U_1\cup V_2)$
which induces the given filtered $G_0$--structure.  Similarly, one
next extends the Cartan connection to $U_1\cup V_2\cup V_3$ and so on,
and in finitely many steps, one obtains a regular Cartan connection on
$\Cal G$ inducing the given given filtered $G_0$--structure. Now part
(1) follows from Theorems \ref{thm4.4} and \ref{thm4.5}.

\smallskip

(2) Let $(p:\Cal G\to M,\om)$ and $(\tilde p:\tcg\to\tilde
M,\tilde\om)$ be regular normal Cartan geometries of type $(\frak
g,P)$ and suppose that $F:\Cal G/P_+\to \tcg/P_+$ is a morphism of
filtered $G_0$--structures with base map $f:M\to\tilde M$. In
particular, this means that $f$ is a local diffeomorphism, see Section
\ref{2.1}. Now we can consider the pullback $f^*p:f^*\tcg\to M$, which
is a principal $P$--bundle. By construction, this comes with a
morphism $p^*f:f^*\tcg\to\tcg$ of principal bundles with base map
$f:M\to\tilde M$.

We can also form the pullback bundle $f^*(\tcg/P_+)\to M$, which is a
principal bundle with structure group $G_0$. By the universal property
of this pullback, we get a map $f^*\tcg\to f^*(\tcg/P_+)$ covering the
identity on $M$, which induces an isomorphism $(f^*\tcg)/P_+\cong
f^*(\tcg/P_+)$. Also by this universal property, the bundle map
$F:\Cal G/P_+\to\tcg/P_+$ induces a homomorphism $\Cal G/P_+\to
f^*(\tcg/P_+)$ which covers the identity on $M$ and thus is an
isomorphism of principal $G_0$--bundles. Now we claim that it suffices
to lift this to a homomorphism $\Cal G\to f^*\tcg$ to complete the
proof.

Having done that, we compose with $p^*f$ to obtain a homomorphism
$\hat F:\Cal G\to \tcg$ of principal bundles with base map $f$, which
by construction lifts $F:\Cal G/P_+\to\tcg/P_+$. Since $\hat F$ is
$P$--equivariant and $f$ is a local diffeomorphism, we readily
conclude that $\hat\om:=\hat F^*\tilde\om$ is a Cartan connection on
$\Cal G$. Naturality of the exterior derivative implies that for the
curvature of $\hat\om$ we get $\hat K=\hat F^*\tilde K$, where $\tilde
K$ is the curvature of $\tilde\om$. This in turn means that the
curvature functions are related by $\hat\ka=\tilde\ka\o\hat F$,
compare with Section 1.5.2 of \cite{book}.

This readily implies that all values of $\hat\ka$ are homogeneous of
degree $\geq 1$ and lie in $\Cal N$, so $\hat\om$ is regular and
normal. Finally, the fact that $\hat F$ lifts $F$, which by assumption
is a morphism of filtered $G_0$--structures exactly says that $\hat
F^*\tilde\om$ and $\om$ induce the same underlying filtered
$G_0$--structure on $M$. Hence we can apply Theorem \ref{thm4.5} to
obtain an automorphism $\Ph:\Cal G\to\Cal G$ of the principal
$P$--bundle $\Cal G$ which induces the identity on $\Cal G/P_+$ such
that $\Ph^*\hat\om=\om$. But this exactly says that $\hat F\o\Ph:\Cal
G\to\tcg$ is a morphism of Cartan geometries lifting $F:\Cal
G/P_+\to\tcg/P_+$.

\smallskip

So it remains to prove the following claim (for which we change the
notation slightly): Suppose that $\Cal G\to M$ and $\tcg\to M$ are
principal $P$--bundles. Then for any homomorphism $\ps:\Cal
G/P_+\to\tcg/P_+$ of principal $G_0$--bundles with base map $\id_M$
there is a lift to a homomorphism $\Ps:\Cal G\to\tcg$ of principal
$P$--bundles.

Observe first that this is true locally on any subset $U\subset M$
over which both $\Cal G$ and $\tcg$ are trivial. Indeed, let
$\si:U\to\Cal G$ and $\tilde\si:U\to\tcg$ be smooth sections and
consider the induced sections $\si_0:U\to\Cal G/P_+$ and
$\tilde\si_0:U\to\tcg/P_+$. Since $\ps$ has base map $\id_M$, there is
a smooth function $g_0:U\to G_0$ such that
$\ps(\si_0(u))=\hat\si_0(u)\cdot g_0(u)$. Now let $\iota:G_0\to P$ be
a smooth homomorphism as in Definition \ref{def4.6}, and define
$\Ps:\Cal G|_U\to\tcg|_U$ by $\Ps(\si(u)\cdot g):=\tilde\si(u)\cdot
(\iota(g_0(u))g)$ for all $u\in U$ and $g\in P$. By construction, this
is $P$--equivariant, has base map $\id_M$ and induces $\ps$ on the
underlying $G_0$--bundles.

Having this result at hand, we can complete the proof as for part (1).
There is a finite open covering $U_1,\dots,U_N$ of $M$ such that both
$\Cal G$ and $\tcg$ are trivial over each $U_i$, so in particular, we
have a section $\si_i$ of $\Cal G$ over each $U_i$. We also choose
open subsets $V_i$ such that $\bar V_i\subset U_i$ and
$M=V_1\cup\dots\cup V_N$. By the last step, we find a lift $\Ps_i$ of
$\ps$ over each $U_i$. Now over $U_{12}:=U_1\cap U_2$, there is a
smooth function $g:U_{12}\to P_+$ such that $\Ps_2(u)=\Ps_1(u)\cdot
g(u)$. Equivariancy of $\Ps_2$ and $\Ps_1$ shows that $g(u\cdot
h)=h^{-1}g(u)h$ for any $h\in P$. Since $P$ is of split exponential
type, we get $g(u)=\exp(Z(u))$ for a smooth function $Z:U_{12}\to\frak
g^1$ such that $Z(u\cdot h)=\Ad(h^{-1})(Z(u))$ for all $h\in P$. Now
let $f$ be a bump function with support contained in $U_2$, which is
identically one on $V_2$. Then we can extend $u\mapsto f(p(u))Z(u)$ by
zero to a smooth function $M\to\frak g^1$, so in particular $u\mapsto
\Ps_1(u)\exp(f(p(u))Z(u))$ is a principal bundle homomorphism over
$U_1$ lifting $\ps$. But over $U_1\cap V_2$, this coincides with
$\Ps_2$, so we can piece them together to obtain a lift of $\ps$
defined on $U_1\cup V_2$. Iterating this finitely many times, we reach
a global lift of $\ps$.
\end{proof}

\end{document}